\documentclass[12pt]{amsart}

\usepackage{enumerate}

\usepackage[foot]{amsaddr}
\usepackage[english]{babel}

\allowdisplaybreaks

\usepackage{xurl}  
\usepackage[%pagebackref,
pdftex,hyperfootnotes]{hyperref}
\hypersetup{
	colorlinks=true,
	linkcolor=NavyBlue, 
	urlcolor=RoyalPurple,
	citecolor=OliveGreen,
	pdftitle={Bivariate Multiple Orthogonal Polynomials of Mixed Type},
	bookmarks=true,
}

\usepackage[usenames,dvipsnames,svgnames,table,x11names]{xcolor}
\usepackage{pgfplots}
\usepackage{tikz}
\usepackage{tikz-3dplot}
\usetikzlibrary{automata,quotes, chains,matrix,calc,shadows,shapes.callouts,shapes.geometric,shapes.misc,positioning,patterns,decorations.shapes,
	decorations.pathmorphing,decorations.markings,decorations.fractals,decorations.pathreplacing,shadings,fadings,arrows.meta,bending}
\usepackage{drawmatrix}
\usepackage[utf8]{inputenc}

\usepackage{comment}
\usepackage{graphicx}
\usepackage[textwidth=17.5cm,textheight=22.275cm,
height=%24cm,
24.275cm,
width=18cm]{geometry}

\usepackage{amssymb,latexsym,amsmath,amsthm,bm}
\usepackage{mathrsfs}
\usepackage{mathtools,arydshln,mathdots}
\mathtoolsset{showonlyrefs}
\usepackage{tcolorbox}
\usepackage{mathtools}

\usepackage[table]{xcolor}

\usepackage{Baskervaldx}
\usepackage[]{newtxmath}

\usepackage[all]{xy}
\hypersetup{
	colorlinks=true,
	linkcolor=blue}
\usepackage{nicematrix}
\let\originalDdots\Ddots % Save original definition
\usepackage{mathdots} % Load mathdots
\let\Ddots\originalDdots % Restore nicematrix's version
\date{}

\theoremstyle{plain}

\newtheorem{Theorem}{Theorem}[section]
\newtheorem{Corollary}[Theorem]{Corollary}
\newtheorem{Lemma}[Theorem]{Lemma}
\newtheorem{Proposition}[Theorem]{Proposition}
\newtheorem{Definition}[Theorem]{Definition}
\newtheorem{Remark}[Theorem]{Remark}
\newtheorem{Example}[Theorem]{Example}

\renewcommand{\d}{\operatorname{d}}

\newcommand{\N}{\mathbb{N}}
\newcommand{\R}{\mathbb{R}}

\newcommand{\astuparrow}{\mathbin{\ast\!\!\uparrow}}

\newcommand{\tdeg}{\operatorname{total-deg}}
\newcommand{\ldeg}{\operatorname{grlex-deg}}
\newcommand{\lpos}{\operatorname{grlex-pos}}
\begin{document}
	
		\title[Bivariate mixed-type multiple orthogonality]{Bivariate multiple orthogonal polynomials  of mixed type\\on the step-line }

	\author[M Mañas]{Manuel Mañas$^{1}$}
	
	\author[M Rojas]{Miguel Rojas$^{2}$}

    \author[J Wu]{JianWen Wu$^3$
    } %Complete
	\address{$^{1,2,3}$Departamento de Física Teórica, Universidad Complutense de Madrid, Plaza Ciencias 1, 28040-Madrid, Spain}
	\address{$^3$Department of Mathematics and Physics,  
		Quzhou University,  
		78 North Jiuhua Road,
		Quzhou City, Zhejiang Province 324000,
		P. R. China}
	
	\email{$^{1}$manuel.manas@ucm.es}
	\email{$^{2}$migroj01@ucm.es}
	\email{$^3$32070@qzc.edu.cn}
	\thanks{$^3$During the preparation of this work, JW held a 7-month postdoctoral visiting position at the Departamento de Física Teórica, Universidad Complutense de Madrid, while on leave from Quzhou University.}

	\keywords{mixed-type multiple orthogonal polynomials, bivariate orthogonal polynomials, LU factorization, recurrence relations, Christoffel--Darboux kernels, Christoffel--Darboux formulas, ABC theorem}

	\begin{abstract}
         This article studies bivariate multiple orthogonal polynomials of the mixed type on the step-line. The analysis is based on the LU factorization of a moment matrix specifically adapted to this framework. The orthogonality and biorthogonality relations satisfied by these polynomials are identified, and their precise multi-degrees are determined. The corresponding recurrence relations and the growing band matrices that encode them are also derived. Christoffel–Darboux kernels and the associated Christoffel–Darboux-type formulas are obtained. An ABC-type theorem is established, relating the inverse of the truncated moment matrix to these kernels. As an illustration, the bivariate Jacobi–Piñeiro multiple orthogonal polynomials of mixed type on the triangle are computed by means of an $LU$ factorization implemented in a dedicated Maple script.
	\end{abstract}
	
	\subjclass{42C05, 33C45, 33C47, 47B39, 47B36}
	\maketitle
\tableofcontents

\section{Introduction}

In this article, we develop a Gauss–Borel factorization-based approach \cite{manas} to the study of mixed-type multiple orthogonal polynomials in two real variables along the step-line. Building on the recent work by Fernández and Villegas \cite{Lidia}, which introduced multiple orthogonal polynomials over the full lattice in two variables, our contribution focuses on the mixed-type on the step-line case and exploits the algebraic structure encoded in the moment matrix to derive orthogonality relations, recurrence formulas, Christoffel–Darboux-type identities, and an ABC-type theorem.

Multiple orthogonal polynomials constitute a rich and adaptable family of functions with wide-ranging applications in both pure and applied mathematics (see \cite{nikishin_sorokin,Aptekarev}). In contrast to their classical counterparts, which rely on a single weight function, multiple orthogonal polynomials are characterized by satisfying orthogonality conditions simultaneously with respect to several weights or measures. These systems have emerged as essential tools in areas such as numerical analysis, approximation theory, and mathematical physics, where they are particularly suited to addressing problems involving multi-component structures. A general overview can be found in \cite{Ismail}.

Within numerical analysis, especially in the development of quadrature formulas of mixed type, a key milestone was Borges' introduction of simultaneous Gaussian quadrature \cite{Borges}. This framework was further expanded by the work of Coussement and Van Assche \cite{Coussement-Van_Assche}, while Van Assche later proposed a Golub–Welsch-type algorithm to generate such quadrature rules \cite{Van_Assche}. These ideas were further linked in \cite{aim} to the theory of positive bidiagonal factorizations of banded matrices, to spectral versions of Favard’s theorem, and to mixed-type simultaneous Gaussian quadrature.

Historically, the theory of Hermite–Padé approximants and constructive function theory has provided a natural setting for the study of multiple orthogonal polynomials. Introductory texts include the foundational book by Nikishin and Sorokin \cite{nikishin_sorokin}, and the dedicated chapter by Van Assche in \cite[Ch. 23]{Ismail}. Their relationship with integrable systems has been explored in \cite{afm}, with a more accessible presentation given in \cite{andrei_walter}. For the asymptotics of zeros, one may consult \cite{Aptekarev_Kaliaguine_Lopez}, while the Gauss–Borel perspective is examined in \cite{afm}, and applications to random matrix theory are found in \cite{Bleher_Kuijlaars}.

Applications of mixed-type multiple orthogonal polynomials are diverse. They appear in problems related to Brownian bridges and non-intersecting Brownian motions \cite{Evi_Arno}, as well as in the theory of multicomponent Toda lattices \cite{adler,afm}. Their relevance in number theory is also well documented: they played a pivotal role in Apéry’s proof of the irrationality of $\zeta(3)$ \cite{Apery}, and in subsequent results on the irrationality of other odd zeta values \cite{Ball_Rivoal,Zudilin}.
The work \cite{ulises} investigated the logarithmic and ratio asymptotic behavior of linear forms arising from Nikishin systems, where the measures themselves are generated by a second Nikishin structure. A broader study of mixed-type multiple orthogonal polynomials, including detailed properties of their zeros, was undertaken in \cite{ulises2}.

More recently, mixed-type multiple orthogonal polynomials have taken center stage in the spectral analysis of banded semi-infinite matrices. This line of research has been developed in \cite{aim, phys-scrip, BTP, Contemporary}, with additional results in \cite{laa}. They also prove instrumental in the analysis of Markov chains and random walks that extend beyond classical birth and death processes, as shown in \cite{CRM,finite,hypergeometric,JP}. Other important results concerning multiple orthogonal polynomials include the derivation of hypergeometric expressions for both type I and type II polynomials, as well as for their recurrence coefficients. In addition, integral representations have been established for all families lying beneath the Hahn polynomials in the multiple Askey scheme; see \cite{ExamplesMOP0,ExamplesMOP,ExamplesMOP2}. These hypergeometric multiple orthogonal polynomials have also been connected with free probability theory, as discussed in \cite{ZerosFree}.

The study of Christoffel and Geronimus perturbations was extended in \cite{bfm} to the framework of multiple orthogonal polynomials with two weights. In that work, explicit connection formulas were derived linking type II multiple orthogonal polynomials, type I linear forms, and associated vector Stieltjes–Markov functions. The matrix polynomials used for perturbation were not assumed to be monic but belonged to a restricted class.  Further progress along these lines was made in \cite{Mañas_Rojas,Mañas_Rojas_0}, where Christoffel and Geronimus transformations were generalized to the setting of mixed-type multiple orthogonal polynomials on the step-line. That study was conducted in a manner closely aligned with the approach adopted here.
These results for the mixed-type case were, in part, motivated by previous developments in the matrix orthogonal polynomial setting. This line of research began with the analysis of Christoffel transformations for monic matrix orthogonal polynomials in \cite{AAGMM}, and was followed by the introduction of spectral techniques for monic Geronimus transformations in \cite{AGMM}, with non-monic cases also considered via non-spectral means. A more detailed treatment of the Geronimus–Uvarov transformation and its relation to non-Abelian Toda lattices appeared in \cite{AGMM2}.

The extension of orthogonal polynomials to several variables has given rise to a rich and active field, with foundational work going back to the seminal paper by Krall and Sheffer \cite{Krall}, where systems of multivariate orthogonal polynomials satisfying partial differential equations were first systematically studied. A thorough development of the general theory can be found in the monograph by Suetin \cite{Suetin}, which focuses on analytical foundations and approximation properties. The book by Dunkl and Xu \cite{Dunkl_Xu} offers a broad and rigorous treatment, emphasizing classical families in several variables, orthogonality with respect to product measures and measures invariant under reflection groups, and connections with special functions.

The study of orthogonal polynomials associated with root systems has  played a fundamental role in the development of multivariate orthogonality. In his influential work \cite{Koornwinder1975}, Koornwinder introduced several families of two-variable analogues of classical orthogonal polynomials, providing a classification that inspired much subsequent research. Later, in \cite{Koornwinder1992}, he extended the theory to the context of Askey–Wilson polynomials for root systems of type  B, C and BC, paving the way for deep connections with special functions and representation theory. A comprehensive framework for such families, including Macdonald and Macdonald–Koornwinder polynomials, is presented in Macdonald’s monograph \cite{Macdonald2000}, which remains a cornerstone in the theory of symmetric functions and orthogonal polynomials associated with root systems.

Yuan Xu has made important recent contributions to multivariate orthogonal polynomials, for example in the areas of connection coefficients, orthogonality on quadratic surfaces, and Gaussian cubature rules. In collaboration with Plamen Iliev, he derived explicit formulas for connection coefficients between classical orthogonal polynomial bases—such as Jacobi, Hahn, and Krawtchouk families—expressed in terms of Tratnik’s multivariable Racah polynomials \cite{Iliev_Xu_2017}. With Sheehan Olver, Xu developed explicit constructions of orthogonal polynomials both inside and on quadratic surfaces of revolution—cones, paraboloids, and hyperboloids—applying these results to cubature and fast numerical approximation methods \cite{Olver_Xu_2020}. Additionally,  \cite{Xu_2015} provided a rigorous method to generate two-variable orthogonal polynomials from near-banded Toeplitz matrices, leading to families with maximal real zeros and associated Gaussian cubature rules.

Significant contributions have also been made in the study of modifications of multivariate orthogonal systems. In particular, the work \cite{Granada_SIGMA} develops a general framework for Christoffel and Uvarov-type transformations of multivariate moment functionals, providing explicit expressions for the orthogonal polynomials and reproducing kernels, with examples on the ball and the simplex. This continues a research line exemplified in earlier studies such as \cite{Granada_Krall}, which explores Krall-type modifications and their impact on multivariate orthogonality.

Recent developments have explored the relationship between multivariate orthogonality and integrable systems. In a series of works \cite{AiM_multivariable,JAT_multivariable,PRIMS}, Ariznabarreta and Mañas developed a matrix approach to multivariate orthogonal and Laurent orthogonal polynomials, establishing links with integrable hierarchies and spectral transformations. These results generalize classical constructions—such as the Christoffel and Darboux transformations—to higher-dimensional settings and reveal the algebraic and analytical structure underlying multivariate orthogonal families.

Further advances have been made in the study of multiple orthogonality in the multivariate setting. Of particular relevance is the recent work by Fernández and Villegas \cite{Lidia}, which served as a source of inspiration for the present article. In their study, they introduce and analyze multiple orthogonal polynomials in two real variables defined over the full lattice. In contrast, the present work focuses on polynomials restricted to the step-line. Moreover, while their construction does not involve mixed-type orthogonality, the families considered here are of mixed type. Their recurrence relations are of nearest-neighbor type on the lattice, whereas ours are adapted to the step-line structure. As a result, both the Christoffel–Darboux formulas and the ABC-type theorem presented here are formulated along the step-line.

The article, apart from this introduction,  is structured in four main sections, each developing a key aspect of the theory of mixed-type multiple orthogonal polynomials in two variables along the step-line. Section 2  introduces the construction of structured matrices associated with suitable monomial sequences in two variables, providing the algebraic foundation for the rest of the work. Section 3 develops the LU factorization of adequate moment matrices and its role in generating bivariate  orthogonal polynomials with respect to mixed-type conditions. Section 4 is devoted to the recurrence relations satisfied by these polynomial families, with special attention to the step-line. In Section 5 Christoffel–Darboux-type kernels are introduced and analyzed in this bivariate setting, leading to a formulation of an ABC-type theorem adapted to the bi-variated mixed-type multiple  context.  Finally, in Section~6 we present a case study on the bivariate Jacobi–Piñeiro multiple orthogonal polynomials of mixed type on the triangle and compute these polynomials explicitly by means of the $LU$ factorization, using a specially designed Maple program available on GitHub; see the Declarations section at the end of the paper.

\section{Semi-infinite matrices and monomial sequences}

We begin by establishing the notation and definitions that will be used throughout this section. Our approach involves decomposing a moment matrix using the $LU$ factorization, which requires working with semi-infinite matrices—i.e., matrices of size $\infty \times \infty$ arranged in $\N_0^2$. See \cite{manas} for an introduction of the $LU$ method in orthogonal polynomial theory.

We will require to give some structure to the matrices we work with. 
\begin{Definition} \label{Def: Decomposition1}
    Let $M$ be some semi-infinite matrix with entries given by:
    \begin{equation*}
        M = 
        \begin{bNiceMatrix}
            M_{0,0} & M_{0,1} & M_{0,2} & \Cdots[shorten-end=7pt] \\
            M_{1,0} & M_{1,1} & M_{1,2} & \Cdots[shorten-end=7pt] \\
            M_{2,0} & M_{2,1} & M_{2,2} & \Cdots[shorten-end=7pt] \\
            \Vdots[shorten-end=4pt] & \Vdots[shorten-end=4pt] & \Vdots[shorten-end=4pt] & 
        \end{bNiceMatrix}.
    \end{equation*}
    We denote by $M_{(i,j)}$ a block of size $r_1 \times r_2$ located at the $(i,j)$-th entry. With this block-wise decomposition, the matrix can be written as:
    \begin{equation*}
        M = 
        \begin{bNiceMatrix}
            M_{(0,0)} & M_{(0,1)} & M_{(0,2)} & \Cdots[shorten-end=7pt] \\
           M_{(1,0)} & M_{(1,1)} & M_{(1,2)} & \Cdots[shorten-end=7pt] \\
            M_{(2,0)} & M_{(2,1)} & M_{(2,2)} & \Cdots[shorten-end=7pt] \\
            \Vdots[shorten-end=4pt] & \Vdots[shorten-end=4pt] & \Vdots[shorten-end=4pt] 
        \end{bNiceMatrix}.
    \end{equation*}
    The block size, or equivalently the numbers $r_1$ and $r_2$ will be specified depending on the situation we are considering.
\end{Definition}

Building on this, as we are interested in the bivariate case we introduce bigger  blocks, of sizes growing as we move in the matrix: 
\begin{Definition} \label{Def: Decomposition2}
    We also introduce an alternative decomposition of the semi-infinite matrix $M$:
    \begin{equation*}
        M = 
        \begin{bNiceMatrix}
            M_{[0,0]} & M_{[0,1]} & M_{[0,2]} & \Cdots[shorten-end=7pt] \\
            M_{[1,0]} & M_{[1,1]} & M_{[1,2]} & \Cdots[shorten-end=7pt] \\
            M_{[2,0]} & M_{[2,1]} & M_{[2,2]} & \Cdots[shorten-end=7pt] \\
            \Vdots[shorten-end=4pt] & \Vdots[shorten-end=4pt] & \Vdots[shorten-end=4pt] 
        \end{bNiceMatrix},
    \end{equation*}
    where the block dimensions increase with their indices. Specifically, $M_{[i,j]}$ has size $(i+1)r_1 \times (j+1)r_2$—that is, it consists of $(i+1)$ rows and $(j+1)$ columns of $r_1 \times r_2$ blocks. As in the previous definition, the values of $r_1$ and $r_2$ will be specified in each situation.
\end{Definition}

\begin{Remark}
    For fixed $r_1 = r_2$, this notation guarantees that the matrix product
    \begin{equation*}
        L_{[i,j]} = M_{[i,k]} N_{[k,j]}
    \end{equation*}
    is well-defined, and the resulting matrix adheres to the prescribed block structure.
\end{Remark}

Following the framework established in \cite{Krall}, we now introduce a graded lexicographical sequence to index the degrees of orthogonal polynomials in two variables.
% \begin{Definition} \label{Def: Special Function}
%     Let us introduce the following functions: 
%     \[
%         \mathcal{F}_1^{\pm}(x) = \left\lfloor \pm \frac{1}{2} + \frac{1}{2} \sqrt{1 + 8x} \right\rfloor,
%     \]
%     where $\left\lfloor \cdot \right\rfloor$ stands for the floor function. We also need to introduce, 
%     \begin{align*}
%         \mathcal{F}_2^{+}(x) & = \mathcal{F}_1^{+}(x) + 1, & \mathcal{F}_2^{-}(x) & = \mathcal{F}_1^{-}(x-1) + 1.
%     \end{align*}
%     We generically denote them by $\mathcal{F}_k^{\pm}(x)$.
% \end{Definition}
\begin{Definition} \label{Def: Sequence 1}
     Consider the graded-lexicographical sequence:
    \begin{equation*}
        s \coloneq \{ (0,0),\, (1,0),\, (1,1),\, (2,0),\, (2,1),\, (2,2),\, \dots,\, (i,0),\, \dots,\, (i,i),\, \dots \},
    \end{equation*}   
    with elements pair of nonnegative integers $(i,j)$ with $i \in \N_0$ and $0 \leq j \leq i$.
\end{Definition}
For the study of  this graded-lexicographical sequence $s$ it is convenient to introduce the following:
 \begin{Definition} \label{Def: Special Function}
Let us define the  following functions: 
	\[
%	\mathcal{F}_1^{\pm}(x) \coloneq \left\lfloor \pm \frac{1}{2} + \frac{1}{2} \sqrt{1 + 8x} \right\rfloor.
	\begin{aligned}
	\mathcal{F}(x) &\coloneq \left\lfloor - \frac{1}{2} + \frac{1}{2} \sqrt{1 + 8x} \right\rfloor, &
			\mathcal{F}_1^{+}(x)&\coloneq	\mathcal{F}(x)+1
	\end{aligned}
	\]
Here, $\left\lfloor x \right\rfloor$ denotes the floor function, that is, the function which assigns to a real number $x$ the greatest integer less than or equal to $x$.
	Additionally, we will use:
%\[	\begin{aligned}
%		\mathcal{F}_2^{+}(x) & \coloneq\mathcal{F}_1^{+}(x) + 1, & \mathcal{F}_2^{-}(x) & \coloneq\mathcal{F}_1^{-}(x-1) + 1.
%	\end{aligned}\]
\[	\begin{aligned}
	\mathcal{F}_2^{+}(x) & \coloneq\mathcal{F}(x) + 2, & \mathcal{F}_2^{-}(x) & \coloneq\mathcal{F}(x-1) + 1.
\end{aligned}\]
	Lastly, we will also use the notation: $\mathcal{F}(x) = \mathcal{F}_1^-(x)$, since it will become convinient to refer to $\mathcal{F}_k^-(x)$ for $k \in \{1,2\}$, jointly.
%	These are generically denoted by $\mathcal{F}_k^{\pm}(x)$.
\end{Definition}

\begin{Proposition}
    The graded lexicographical sequence $s$ defines a bijection between $\N_0$ and the pairs $(i,j)$ with $i \in \N_0$ and $0 \leq j \leq i$. For a given pair $(i,j)$, its position $I\in\N_0$ in the sequence is given by:
    \begin{equation*}
        I = \frac{i(i+1)}{2} + j.
    \end{equation*}
    Conversely, for a given position $I\in \N_0$, the corresponding pair $(i,j)\in s$  in the lexicographical sequence is determined by:
\[    \begin{aligned}
        i & = \mathcal F(I), &
        j & = I - \frac{i(i+1)}{2},
    \end{aligned}\]

\end{Proposition}

\begin{Definition}
    To refer simultaneously to both the position and the corresponding pair of numbers, we adopt the notation:
    \begin{equation*}
        I \sim (i,j).
    \end{equation*}
\end{Definition}

\begin{Proposition} \label{Prop: Equiv Decom}
    The relationship between the two matrix decompositions (Definitions \ref{Def: Decomposition1} and \ref{Def: Decomposition2}) is described by:
    \begin{equation*}
        {M}_{[i,j]} = \begin{bNiceMatrix}
            {M}_{(I,J)} & \Cdots & {M}_{(I,J+j)} \\
            \Vdots & & \Vdots \\
            {M}_{(I+i,J)} & \Cdots & M_{(I+i,J+j)}
        \end{bNiceMatrix},
    \end{equation*}
    where $I \sim (i,0)$ and $J \sim (j,0)$. Recall that each block ${M}_{(k,l)}$ has dimensions $r_1 \times r_2$.
\end{Proposition}

We now introduce a matrix of monomials in two variables $x_1,x_2$ that will be essential in the construction of the moment matrix.  Identity and zero matrices are denoted by \( I_r ,0_r \in\R^{r\times r}\), respectively.

\begin{Definition}
    For $r\in\N$, the semi-infinite monomial matrix of size $\infty \times r$ is defined as:
    \begin{equation*}
        X_{[r]}(\boldsymbol{x}) \coloneq \begin{bNiceMatrix}
            I_r & x_1 I_r & x_2 I_r & x_1^2 I_r & x_1x_2 I_r & x_2^2 I_r & \Cdots & x_1^i I_r & \Cdots & x_1^{i-j} x_2^j I_r & \Cdots & x_2^i I_r & \Cdots[shorten-end=7pt]  
        \end{bNiceMatrix}^{\top},
    \end{equation*}
    where $0 \leq j \leq i$. The monomials $x^{i-j} y^j$ are ordered according to the graded lexicographic sequence $s$ and the associated pair $(i,j)$ introduced in Definition \ref{Def: Sequence 1}. The position of each block corresponds to the index $I$ (with $I \sim (i,j)$), and each block is of size $r$.
\end{Definition}

Associated with these monomial matrices we have the following sequences that will be useful later on. These numbers  help in the  characterization of the rows in \( X_{[r]}(\boldsymbol{x}) \) where nontrivial powers of \( x_k \) appear.

\begin{Definition}
    Let \( r \in \mathbb{N} \), and consider the sequence:
\[
        s_{r} = \big\{ (1,0),\, (1,1),\, \dots,\, (1,r-1),\, (2,0),\, \dots,\, (2,2r-1),\, \dots,\, (i,0),\, \dots,\, (i,ir-1),(i+1,0)\, \dots \big\},
\]
    where \( i \in \mathbb{N} \) and \( j \in \{ 0, \dots, ir - 1 \} \). To this sequence, we associate three natural numbers:
\[    \begin{aligned}
        n^{-}_r & \coloneq r \frac{i(i - 1)}{2} + j, &
        n^{+}_{r;1} & \coloneq r \frac{i(i+ 1)}{2} + j, &
        n^{+}_{r;2} & \coloneq r \frac{i(i + 1)}{2} + j + r. 
    \end{aligned}\]
\end{Definition}

\begin{Proposition} \label{Prop: Biyection with N_0}
  For $r\in\N$,   the map $n_r^-$ defines a bijection between \( s_r \) and \( \mathbb{N}_0 \). Its inverse is given by:
    \[
\begin{aligned}
	        i &= \mathcal{F}\left(\frac{n_r^-}{r}\right)+1 = \mathcal{F}_1^+\left(\frac{n_r^-}{r}\right),
        %\left\lfloor \frac{1}{2} + \frac{1}{2} \sqrt{1 + \frac{8n^-_{r}}{r}} \right\rfloor,      
        &
        j &= n^-_{r} - r\frac{i(i+1)}{2}.
\end{aligned}
    \]
\end{Proposition}

\begin{proof}
%	\textcolor{red}{\textbf{who is $n'$?}I've made changes consequently}
    For consecutive pairs \( (i,j) \) and \( (i,j+1) \), we have 
\(n^-_r(i,j+1)-n^-_r(i,j)=1\)    
   % \( (n')^-_r - n^-_r = 1 \). 
    When transitioning from \( (i, ir - 1) \) to \( (i+1, 0) \), we compute:
    \[
      %  (n')^-_r - n^-_r = 
    n^-_r(i+1,0)-  n^-_r(i,ir-1)=
        r\frac{i(i+1)}{2} - \left(r\frac{i(i-1)}{2} + ir - 1\right) = 1.
    \]
    Since \( n^-_r (1,0) = 0 \), this confirms the bijection. The inverse follows by solving the quadratic relation and selecting the integer root via the floor function.
\end{proof}
Now,  with the characterization of the rows in \( X_{[r]}(\boldsymbol{x}) \) where nontrivial powers of \( x_k \), 
we consider the following sequences.

\begin{Definition}
    Define two sequences:
    \begin{align*}
        J_{r;1} &\coloneq \{0,\, \dots,\, r-1,\, 2r,\, \dots,\, 3r-1,\, 5r,\, \dots,\, 6r-1,\, 9r,\, \dots,\, 10r-1,\, \dots\}, \\
        J_{r;2} &\coloneq \{0,\, \dots,\, r-1,\, r,\, \dots,\, 2r-1,\, 3r,\, \dots,\, 4r-1,\, 6r,\, \dots,\, 7r-1,\, \dots\}.
    \end{align*}
\end{Definition}

 Both sequences are constructed from $\N_0$  by erasing blocks of increasing size and keeping blocks of size $r$.
 	%  Both sequences are characterized by growing included blocks and fixed-size excluded intervals of size \( r \). 
	For the first sequence, the particular partition  of $\N_0$ in terms of these erased and kept blocks is as follows:
	\begin{multline*}
		\mathbb{N}_0 = \big\{ \underbracket[1pt][2pt]{0,\, \dots ,\, r-1}_{\subset J_{r;1}}, \, \underbracket[1pt][2pt]{r,\, \dots ,\, 2r-1}_{\subset \mathbb{N}_0 \setminus J_{r;1}}, \, \underbracket[1pt][2pt]{2r,\, \dots ,\, 3r-1}_{\subset J_{r;1}}, \underbracket[1pt][2pt]{3r,\, \dots ,\, 4r-1,\, 4r,\, \dots ,\,  5r-1}_{\subset \mathbb{N}_0 \setminus J_{r;1}} \\ 
		\underbracket[1pt][2pt]{5r,\, \dots ,\, 6r-1}_{\subset J_{r;1}}, \, \underbracket[1pt][2pt]{6r,\, \dots ,\, 7r-1, \, 7r,\, \dots ,\, 8r-1,  8r,\, \dots ,\, 9r-1}_{\subset \mathbb{N}_0 \setminus J_{r;1}},\, \underbracket[1pt][2pt]{9r,\, \dots ,\, 10r-1}_{\subset J_{r;1}}, \, \dots \big\}.
	\end{multline*}
	For the second sequence, it looks as follows
	\begin{multline*}
		\mathbb{N}_0 = \big\{ \underbracket[1pt][2pt]{0,\, \dots, \, r-1, \, r,\, \dots ,\, 2r - 1}_{\subset J_{r;2}}, \, \underbracket[1pt][2pt]{2r,\, \dots ,\, 3r-1}_{\subset \mathbb{N} \setminus J_{r;2}}, \, \underbracket[1pt][2pt]{3r,\, \dots ,\, 4r-1}_{\subset J_{r;2}}, \underbracket[1pt][2pt]{4r,\, \dots ,\, 5r-1,\, 5r,\, \dots ,\,  6r-1}_{\subset \mathbb{N} \setminus J_{r;2}}, \\ 
		\underbracket[1pt][2pt]{6r,\, \dots ,\, 7r-1}_{\subset J_{r;2}}, \, \underbracket[1pt][2pt]{7r,\, \dots ,\, 8r-1, \, 8r,\, \dots ,\, 9r-1,  9r,\, \dots ,\, 10r-1}_{\subset \mathbb{N} \setminus J_{r;2}}, \, \dots \big\}.
	\end{multline*}

\begin{Proposition} \label{Prop: Bijection with NJ}
    The maps $n^+_{r;k}$ define bijections from $s_r$ to $\mathbb{N}_0 \setminus J_{r;k}$ for $k \in \{1,2\}$. Their inverses are given by:
\[    \begin{aligned}
        i &= %\left\lfloor -\frac{1}{2} + \frac{1}{2} \sqrt{1 + \frac{8n^+_{r;1}}{r}} \right\rfloor, 
       \mathcal F\left( \frac{n_{r;1}^+}{r} \right), &
        j &= n^+_{r;1} - r\frac{i(i+1)}{2}, \\
        i &= %\left\lfloor -\frac{1}{2} + \frac{1}{2} \sqrt{1 + \frac{n^+_{r;2}}{r}} \right\rfloor, 
        \mathcal F\left(\frac{n_{r;2}^+}{r}-1\right) ,  &
        j &= n^+_{r;2} - r\frac{i(i+1)}{2} - r.
    \end{aligned}\]
\end{Proposition}
\begin{proof}
    For \( k = 1 \), the map increases by one for consecutive values of \( j \), and by \( r \) when transitioning from \( (i, ir-1) \) to \( (i+1, 0) \). The image skips blocks of size \( r \), and since \( j \) grows with \( i \), the definition of \( J_{r;1} \) is satisfied. The case \( k=2 \) follows analogously.
\end{proof}

\begin{Proposition}
    The maps \( \mathbb{N}_0 \rightarrow \mathbb{N}_0 \setminus J_{r;k} \), for \( k \in\{ 1,2 \}\),
    \[
       n^+_{r;k} = n^-_r + r \mathcal{F}_k^+\left(\frac{n_r^-}{r}\right) = n^-_r + r \mathcal{F}\left(\frac{n_r^-}{r}\right) + kr,
       %\begin{aligned}
       	% n^+_{r;1} &= n^-_r + r \mathcal{F}\left(\frac{n_r^-}{r}\right)+r,&
       	 % & n^+_{r;2} &= n^-_r + r \mathcal{F}\left(\frac{n_r^-}{r}\right)+2r,
       	% &\end{aligned}
    \]
    are bijections. Their inverses are given by: 
    \[
        n_r^- = n_{r;k}^+ - r\mathcal{F}^-_k\left( \frac{n_{r;k}^+}{r} \right),
        \]
   That is,
   \[
   \begin{aligned}
   	  n_r^- & = n_{r;1}^+ - r\mathcal{F}\left( \frac{n_{r;1}^+}{r} \right), &  n_r^- &= n_{r;2}^+ - r\mathcal{F}\left( \frac{n_{r;2}^+}{r} -1\right)-r.
   \end{aligned}
    \]
\end{Proposition}
\begin{proof}
For $k=1$ it follows that:
    \[
        n^+_{r;1} - n^-_r = r i.
    \]
    Hence, %\( i = \mathcal{F}^+_1 \left( \frac{n_r^-}{r} \right) = \mathcal{F}^-_1\left( \frac{n_{r;k}^+}{r} \right) \)
\( i = \mathcal{F}\left( \frac{n_r^-}{r} \right) +1= \mathcal{F}\left( \frac{n_{r;k}^+}{r}-1 \right) +1\). The case \( k = 2 \) follows similarly, noting that 
    \[
         n^+_{r;2} - n^-_r = r (i+1).  
    \]
    The bijective nature of these maps follows from the composition of bijective functions.
\end{proof}
As said, the relevance of the sequences \( \mathbb{N}_0 \setminus J_{r;k} \), for \( k  \in\{ 1,2 \}\), lies in their characterization of the rows in \( X_{[r]}(\boldsymbol{x}) \) where nontrivial powers of \( x_k \) (excluding \( x_k^0 \)) appear. For example, for \( k = 1 \), recalling that:
\begin{equation*}
         X_{[r]}(\boldsymbol{x}) = \begin{bNiceMatrix}
            I_r & x_1 I_r & x_2 I_r & x_1^2 I_r & x_1x_2 I_r & x_2^2 I_r & \Cdots[shorten-end=7pt]  
        \end{bNiceMatrix}^{\top},
\end{equation*}
we can check that the erased blocks are $\{ \left[I_r\right],\, \left[ x_2 I_r\right], \, \left[ x_2^2I_r\right], \, \dots \}$, corresponding to rows: \[\{0,\, \dots,\, r-1,\, 2r,\, \dots,\, 3r-1,\, 5r,\, \dots ,\, 6r-1, \, \dots \, \} = J_{r;1}.\]

Let us discuss two important matrices in the discussion of the extended Hankel-type symmetry of the moment matrices we are going to consider below.
\begin{Definition}
    For \( k  \in\{ 1,2 \}\), the shift operators \( \Lambda_{[r];k} \) are defined by:
    \[
        \Lambda_{[r];k} =
        \begin{bNiceMatrix}
            0_{[0,0]} & \Lambda_{[0,1];k} & 0_{[0,2]} &0_{[0,3]} &\Cdots \\[5pt]
            0_{[1,0]} & 0_{[1,1]} & \Lambda_{[1,2];k} & 0_{[1,3]} & \Cdots \\
            0_{[2,0]} & 0_{[2,1]} & 0_{[2,2]} & \Lambda_{[2,3];k} &\Ddots\\
            \Vdots[shorten-end=-10pt] & \Ddots[shorten-end=12pt] & \Ddots [shorten-end=15pt] &\Ddots[shorten-end=7pt]  & \Ddots\\
         \phantom{i}   &         \phantom{i}   &         \phantom{i}   &         \phantom{i}   &         \phantom{i}   
        \end{bNiceMatrix},
    \]
    where
\[    \begin{aligned}
        \Lambda_{[i,i+1];1} &=
        \begin{bNiceMatrix}
            I_r & 0_r & \Cdots & 0_r&0_r \\
            0_r & I_r & \Ddots & \Vdots&\Vdots \\
           \Vdots &\Ddots & \Ddots & 0_r & 0_r \\
            0_r & \Cdots & 0_r & I_r&0_r
        \end{bNiceMatrix}\in\R^{(i+1)r\times (i+2)r}, &
        \Lambda_{[i,i+1];2} &=
        \begin{bNiceMatrix}
           0_r & I_r & 0_r&\Cdots & 0_r \\
     0_r & 0_r & \Ddots& \Ddots& \Vdots\\
     \Vdots &\Ddots & \Ddots &&0_r \\
     0_r & \Cdots & 0_r & 0_r&I_r
        \end{bNiceMatrix}\in\R^{(i+1)r\times (i+2)r}, 
    \end{aligned}\]
    and \( 0_{[i,j]} \) denotes a zero block of size \( (i+1)r \times (j+1)r \).
\end{Definition}

In sparse or semi-infinite matrices such as \( \Lambda_{[r];k} \), the positions or sizes of zero blocks may be omitted if clear from context.

These matrices are relevant becuase they satisfy the following important spectral property in relation with the previously discussed monomial matrices.
\begin{Proposition}
    The shift operators satisfy the eigenvalue-like relations:
    \[
   \begin{aligned}
   	     \Lambda_{[r];1} X_{[r]}(\boldsymbol{x}) &= x_1 X_{[r]}(\boldsymbol{x}), &  \Lambda_{[r];2} X_{[r]}(\boldsymbol{x}) &= x_2 X_{[r]}(\boldsymbol{x}).
   \end{aligned}
    \]
\end{Proposition}
\begin{proof}
	Is a consequence of
	\[
	\begin{aligned}
		\Lambda_{[i,i+1];1}\begin{bNiceMatrix}
			x_1^{i+1} I_r\\[2pt]x_1^i x_2 I_r\\\Vdots \\x_1x_2^i I_r\\[2pt]x_2^{i+1} I_r
		\end{bNiceMatrix}&=x_1 \begin{bNiceMatrix}
			x_1^{i} I_r\\[2pt]x_1^{i-1}x_2 I_r\\\Vdots \\x_2^i I_r
		\end{bNiceMatrix}, &
		\Lambda_{[i,i+1];2}\begin{bNiceMatrix}
			x_1^{i+1} I_r\\[2pt]x_1^i x_2 I_r\\\Vdots \\x_1x_2^i I_r\\[2pt]x_2^{i+1} I_r
		\end{bNiceMatrix}&=x_2 \begin{bNiceMatrix}
			x_1^{i} I_r\\[2pt]x_1^{i-1}x_2 I_r\\\Vdots \\x_2^i I_r
		\end{bNiceMatrix}.
	\end{aligned}
	\]
%    For $k=1$, the shift operator takes the form:
%    \begin{equation*}
%        \Lambda_{[r];1} = \begin{bNiceMatrix}
%            0_r & I_r & 0_r & \Cdots \\
%            0_r & 0_r & 0_r & I_r & 0_r & \Cdots \\
%            0_r & 0_r & 0_r & 0_r & I_r & 0_r & \Cdots \\
%            \Vdots & & & & & \ddots & \ddots 
%        \end{bNiceMatrix}.
%    \end{equation*}
%    Acting on the monomial matrix, we obtain:
%    \begin{equation*}
%        \Lambda_{[r];1} X_{[r]}(\boldsymbol{x}) = \begin{bNiceMatrix}
%            0_r & I_r & 0_r & \Cdots \\
%            0_r & 0_r & 0_r & I_r & 0_r & \Cdots \\
%            0_r & 0_r & 0_r & 0_r & I_r & 0_r & \Cdots \\
%            \Vdots & & & & & \ddots & \ddots 
%        \end{bNiceMatrix} \begin{bNiceMatrix}
%            I_r \\ xI_r \\ yI_r \\ x^2 I_r \\ xy I_r \\ \Vdots[shorten-end=5pt] 
%        \end{bNiceMatrix} = \begin{bNiceMatrix}
%            xI_r \\ x^2 I_r \\ xy I_r \\ \Vdots[shorten-end=5pt] 
%        \end{bNiceMatrix} = x X_{[r]}(\boldsymbol{x}).
%    \end{equation*}
%    The case for $k=2$ follows analogously.
\end{proof}
\begin{Proposition} \label{Prop: Position of 1s Lambda}
    The positions of the \(1\)s in \( \Lambda_{[r];k} \), with \(k \in \{1,2\}\), are given by the pairs \( (n, n^+_{r;k} ) \), where \(n \in \N_0\), \(n^+_{r,k} = n + r\mathcal{F}\left( \frac{n}{r} \right)+rk =  n + r\mathcal{F}_k^+ \left( \frac{n}{r} \right) \) ,  \(k\in\{1,2\}\).
\end{Proposition}

\begin{proof}
    Each row, indexed by \( n \in \N_0 \), contains exactly one \(1\), although certain columns may contain none. From the identity \( \Lambda_{[r];k} X_{[r]}(\boldsymbol{x}) = x_k X_{[r]}(\boldsymbol{x}) \), it follows that the positions of the \(1\)s in each column must lie in the set \( \N_0 \setminus J_{r;k} \). These positions are precisely given by \( n^+_{r;k} \).
\end{proof}

\begin{Proposition} \label{Prop: Position of 1s TLambda}
    The transpose of \( \Lambda_{[r];k} \), denoted by \( \Lambda_{[r];k}^\top \), is given by:
        \[
    \Lambda^\top_{[r];k} =
    \begin{bNiceMatrix}
    	0_{[0,0]} &0_{[0,1]} & 0_{[0,2]} &0_{[0,3]} &\Cdots &\\[5pt]
    \Lambda_{[1,0];k}  & 0_{[1,1]} &   0_{[1,2]} & 0_{[1,3]} & \Cdots& \\
    	0_{[2,0]} &\Lambda_{[1,2];k}& 0_{[2,2]} & 0_{[2,3]} &\Ddots&\\
    		0_{[3,0]} &	0_{[3,1]}& \Lambda_{[3,2];k} &0_{[3,3];k} &\Ddots&\\
    	\Vdots[shorten-end=-10pt] & \Ddots[shorten-end=12pt] & \Ddots [shorten-end=10pt]& \Ddots [shorten-end=-5pt] &\Ddots[shorten-end=-7pt]  & \Ddots\\
    	\phantom{i}   &         \phantom{i}   &         \phantom{i}   &         \phantom{i}   &         \phantom{i}   &         \phantom{i}   
    \end{bNiceMatrix},
    \]
    where \( \Lambda_{[i+1,i];k} = \Lambda_{[i,i+1];k}^\top \). The positions of the \(1\)s are given by the pairs \( (n^+_{r;k}, n) \), with \(k \in \{1,2\}\), \(n \in \N_0\), \(n^+_{r;k} = n + r\mathcal{F}\left( \frac{n}{r} \right)+rk =  n + r\mathcal{F}_k^+ \left( \frac{n}{r} \right)\).
\end{Proposition}
We will omit the reference to the Definitions and Propositions already introduced, since they will be used thoroughly throughout the text.
\section{LU Factorization and Orthogonal Polynomials}
We now delve into the study of mixed multiple orthogonal polynomials in two real variables. To this end, we consider a matrix of measures of size \( q \times p \), defined on a compact set \( \Delta \subset \mathbb{R}^2 \), and assume that each measure is of bounded variation on \( \Delta \). Our approach follows the results presented in \cite{aim} and \cite{BTP}. Many of the properties we will discuss were previously established for the univariate case.

% {\color{blue} Better explanation about the measures we are considering.}

\begin{Definition}
    The Lebesgue–Stieltjes matrix associated with the measure matrix is defined as
    \begin{equation*}
        \d \mu (\boldsymbol{x}) = \begin{bNiceMatrix}
            \d \mu_{1,1}(\boldsymbol{x}) & \Cdots & \d \mu_{1,p}(\boldsymbol{x}) \\
            \Vdots & & \Vdots \\
            \d \mu_{q,1}(\boldsymbol{x}) & \Cdots & \d \mu_{q,p}(\boldsymbol{x})
        \end{bNiceMatrix},
    \end{equation*}
    with support on \( \Delta \).
\end{Definition}

Using the measure matrix, we construct the moment matrix, whose entries consist of integrals of monomials in \( x_1 \) and \( x_2 \) with respect to the components of \( \d\mu(x_1,x_2) \). To ensure a robust analytical framework, we assume all such integrals are finite.

\begin{Definition}
    The moment matrix is given by
    \begin{align*}
        \mathscr{M} & = \int_\Delta X_{[q]}(\boldsymbol{x}) \d \mu (\boldsymbol{x}) X_{[p]}^\top(\boldsymbol{x}) \\
        & =\left[ \begin{NiceMatrix}
            \int_\Delta \d \mu (\boldsymbol{x}) & \int_\Delta  x_1 \d \mu (\boldsymbol{x}) & \int_\Delta  x_2 \d \mu (\boldsymbol{x}) & \int_\Delta  x_1^2 \d \mu (\boldsymbol{x}) & \int_\Delta  x_1x_2 \d \mu (\boldsymbol{x}) & \int_\Delta  x_2^2 \d \mu (\boldsymbol{x}) & \Cdots \\[2pt]
            \int_\Delta x_1\d \mu (\boldsymbol{x}) & \int_\Delta  x_1^2 \d \mu (\boldsymbol{x}) & \int_\Delta  x_1x_2 \d \mu (\boldsymbol{x}) & \int_\Delta  x_1^3 \d \mu (\boldsymbol{x}) & \int_\Delta  x_1^2x_2 \d \mu (\boldsymbol{x}) & \int_\Delta  x_1x_2^2 \d \mu (\boldsymbol{x}) & \Cdots \\[2pt]
            \int_\Delta x_2 \d \mu (\boldsymbol{x}) & \int_\Delta  x_1x_2 \d \mu (\boldsymbol{x}) & \int_\Delta  x_2^2 \d \mu (\boldsymbol{x}) & \int_\Delta  x_1^2x_2 \d \mu (\boldsymbol{x}) & \int_\Delta  x_1x_2^2 \d \mu (\boldsymbol{x}) & \int_\Delta  x_2^3 \d \mu (\boldsymbol{x}) & \Cdots \\
            \Vdots[shorten-end=3pt] &  \Vdots[shorten-end=3pt] &  \Vdots[shorten-end=3pt]  &  \Vdots[shorten-end=3pt]  &  \Vdots[shorten-end=3pt] &  \Vdots[shorten-end=3pt]  & \\\\
        \end{NiceMatrix}\right].
    \end{align*}
    Each entry \( \int_\Delta x_1^{i-j}x_2^j \d \mu (\boldsymbol{x}) \) is a block matrix of size \( q \times p \).
\end{Definition}

\begin{Proposition}
    For two pairs of natural numbers, \( I \sim (i,j) \) and \( K \sim (k,l) \), the moment matrix entries satisfy
    \begin{equation*}
        \mathscr{M}_{(I,K)} = \int_\Delta x_1^{i-j}x_2^j \d \mu(\boldsymbol{x}) x_1^{k-l}x_2^l.
    \end{equation*}
    It follows that \( \mathscr{M}_{(I,K)} = \mathscr{M}_{(K,I)} \), so the moment matrix is block-wise symmetric.
\end{Proposition}

If all leading principal minors of \( \mathscr{M} \) are nonzero, then the matrix admits a Gauss–Borel (LU) factorization: \( \mathscr{M} = \mathscr{L}^{-1} \mathscr{U}^{-1} \), where \( \mathscr{L} \) and \( \mathscr{U} \) are lower and upper triangular matrices, respectively. This factorization is not unique. For instance, if \( \mathscr{D} \) is an invertible diagonal matrix, then \( \tilde{\mathscr{L}} = \mathscr{L}^{-1} \mathscr{D}^{-1} \) and \( \tilde{\mathscr{U}} = \mathscr{D}\mathscr{U}^{-1} \) also yield a valid factorization.

By requiring that \( S \) and \( \bar{S} \) be lower unitriangular matrices (with ones in the main diagonal) and \( H \) be diagonal and invertible, we obtain a unique Gauss–Borel factorization:
\[
\mathscr{M} = S^{-1} H \bar{S}^{-\top}.
\]
Note that in the classical case of a symmetric moment matrix, \( S = \bar{S} \), and a Cholesky factorization holds. However,   this no longer holds in our extended general setting.

Using the moment matrix definition, we derive:
\begin{align*}
    \int_\Delta X_{[q]}(\boldsymbol{x}) \d \mu (\boldsymbol{x}) X_{[p]}^\top(\boldsymbol{x}) & = S^{-1} H \bar{S}^{-\top}, \\
    \int_\Delta S X_{[q]}(\boldsymbol{x}) \d \mu (\boldsymbol{x}) X_{[p]}^\top(\boldsymbol{x}) \bar{S}^\top & = H.
\end{align*}

This factorization provides a natural setting for defining orthogonal polynomials.

\begin{Definition} \label{Def: OrthogonalPoly}
    Define two families of orthogonal polynomials:
    \begin{align*}
        B(\boldsymbol{x}) & = H^{-1}S X_{[q]}(\boldsymbol{x}) = \begin{bNiceMatrix}
            B_{(0,0)}(\boldsymbol{x}) \\ B_{(1,0)}(\boldsymbol{x}) \\ B_{(1,1)}(\boldsymbol{x}) \\ \Vdots[shorten-end=5pt]
        \end{bNiceMatrix} = \begin{bNiceMatrix}
            B_0^{(1)}(\boldsymbol{x}) & \Cdots & B_0^{(q)}(\boldsymbol{x}) \\[3pt]
            B_1^{(1)}(\boldsymbol{x}) & \Cdots & B_1^{(q)}(\boldsymbol{x}) \\[3pt]
            B_2^{(1)}(\boldsymbol{x}) & \Cdots & B_2^{(q)}(\boldsymbol{x}) \\
            \Vdots[shorten-end=5pt] & & \Vdots[shorten-end=5pt]
        \end{bNiceMatrix}, \\
        A(\boldsymbol{x}) & = X_{[p]}^\top(\boldsymbol{x}) \bar{S}^\top = \begin{bNiceMatrix}
            A_{(0,0)}(\boldsymbol{x}) & A_{(1,0)}(\boldsymbol{x}) & A_{(1,1)}(\boldsymbol{x}) & \Cdots[shorten-end=5pt]
        \end{bNiceMatrix} = \begin{bNiceMatrix}
            A_0^{(1)}(\boldsymbol{x}) & A_1^{(1)}(\boldsymbol{x}) & A_2^{(1)}(\boldsymbol{x}) & \Cdots[shorten-end=5pt] \\
            \Vdots & \Vdots & \Vdots \\
            A_0^{(p)}(\boldsymbol{x}) & A_1^{(p)}(\boldsymbol{x}) & A_2^{(p)}(\boldsymbol{x}) & \Cdots[shorten-end=5pt]
        \end{bNiceMatrix},
    \end{align*}
    where \( B_{(n,l)}(\boldsymbol{x}) \) and \( A_{(n,l)}(\boldsymbol{x}) \) are block matrices of size \( q \times q \) and \( p \times p \), respectively. The notation \( B_n^{(b)}(\boldsymbol{x}) \) and \( A_n^{(a)}(\boldsymbol{x}) \) refers to the individual entries.
\end{Definition}

\begin{Proposition}
    The entries of the matrices \( B_{(n,l)}(\boldsymbol{x}) \) and \( A_{(n,l)}(\boldsymbol{x}) \) are given by
    \begin{align*}
        B_{(n,l)}(\boldsymbol{x}) & = \begin{bNiceMatrix}
            B_{Nq}^{(1)}(\boldsymbol{x}) & \Cdots & B_{Nq}^{(q)}(\boldsymbol{x}) \\
            \Vdots & & \Vdots \\
            B_{Nq+q-1}^{(1)}(\boldsymbol{x}) & \Cdots & B_{Nq+q-1}^{(q)}(\boldsymbol{x})
        \end{bNiceMatrix}, &
        A_{(n,l)}(\boldsymbol{x}) & = \begin{bNiceMatrix}
            A_{Np}^{(1)}(\boldsymbol{x}) & \Cdots & A_{Np+p-1}^{(1)}(\boldsymbol{x}) \\
            \Vdots & & \Vdots \\
            A_{Np}^{(p)}(\boldsymbol{x}) & \Cdots & A_{Np+p-1}^{(p)}(\boldsymbol{x})
        \end{bNiceMatrix},
    \end{align*}
    where \( N \sim (n,l) \).
\end{Proposition}
\begin{Definition}
	We refer to the total degree, $\tdeg$, of a bivariate polynomial as the maximum sum of the exponents of the variables $x_1$ and $x_2$ in any of its monomial terms, i.e., $\tdeg x^{n-l}y^l = n$.
\end{Definition}
\begin{Definition}
    We refer to the graded lexicographic degree, grlex-degree, of a bivariate polynomial using the pair \( (n,l) \) associated with $x^{n-l}y^l$, i.e., $\ldeg x^{n-l} y^l = (n,l)$. Alternatively, we can refer to the graded lexicographic position, $\lpos$, by a scalar $N$. Both definitions are related by: \( N \sim (n,l) \). The preceding grlex-position is denoted by \( N - 1 \). However, there are two possible pairs of numbers associated with \( N - 1 \): either \( (n, l - 1) \), if $l \neq 0$, or \( (n - 1, n - 1) \), when $l=0$. When there is no risk of ambiguity, we will generically represent the preceding grlex-degree as \( (n, l - 1) \).
\end{Definition}
\begin{Remark}
	The graded lexicographic degree satisfies the property that the degree of the product of two bivariate polynomials equals the sum of their individual degrees. However, this property does not hold for the grlex-position. With this in mind—and when no confusion arises—we will refer to both the degree and the position interchangeably.
\end{Remark}

\begin{Proposition}[Graded lexicographic degree structure of orthogonal polynomials] \label{Prop: DegreeStructure}
    The matrix polynomials \( B_{(n,l)}(\boldsymbol{x}) \) and \( A_{(n,l)}(\boldsymbol{x}) \) have diagonal entries of degree \(N \sim (n,l) \), corresponding to \( x^{n-l} y^l \). The diagonal entries of \( A(\boldsymbol{x}) \) are monic, whereas those of \( B(\boldsymbol{x}) \) are non-zero but not necessarily monic. 

    For \( B_{(n,l)}(\boldsymbol{x}) \), entries below the diagonal have degree at most $N\sim (n,l)$, i.e., \( x^{n-l}y^l \), and entries above the diagonal have degree at most $N-1\sim (n,l-1)$, i.e., \( x^{n-l+1}y^{l-1} \). Conversely, for \( A_{(n,l)}(\boldsymbol{x}) \), the higher degree terms lie above the diagonal, and terms below the diagonal are of degree at most $N-1\sim (n,l-1)$, i.e., \( x^{n-l+1}y^{l-1} \).

    Entry-wise, the grlex-position satisfies:
    \begin{align*}
        \lpos B_n^{(b)}(\boldsymbol{x}) & \leq \left\lceil \frac{n+2-b}{q} \right\rceil - 1, &
        \lpos A_n^{(a)}(\boldsymbol{x}) & \leq \left\lceil \frac{n+2-a}{p} \right\rceil - 1.
    \end{align*}
    Equality holds when $n$ can be written as, \( n = Mp + a - 1 \) for \( a \in \{1,\dots,p\} \), and \( n = Mq + b - 1 \) for \( b \in \{1,\dots,q\} \), with \( M \in \mathbb{N}_0 \). Once the grlex-position is known, the degree can be obtained by the bijective relation denoted as $\sim$.
\end{Proposition}

\begin{proof}
    The first part follows directly from the definition. For the \( B(\boldsymbol{x}) \) polynomials, write \( \mathscr{L} = H^{-1}S \), which is an invertible lower triangular matrix. We compute:
    \begin{align*}
        B(\boldsymbol{x}) = \begin{bNiceMatrix}
               \mathscr{L}_{(0,0)} & 0 & 0 & 0 & \Cdots[shorten-end=-5pt] &\phantom{i}  \\
             \mathscr{L}_{(1,0)}&    \mathscr{L}_{(1,1)} & 0 & 0 & \Cdots[shorten-end=-5pt] &\phantom{i}  \\
              \mathscr{L}_{(2,0)}&    \mathscr{L}_{(2,1)}&    \mathscr{L}_{(2,2)} & 0 & \Cdots[shorten-end=-5pt] &\phantom{i}  \\
            \Vdots[shorten-end=-8pt] & \Vdots[shorten-end=-8pt]  & \Vdots[shorten-end=-8pt]  & \Ddots[shorten-end=-17pt]  & \Ddots[shorten-end=-5pt] &\phantom{i}  
            \\
            \phantom{i}   &         \phantom{i}   &         \phantom{i}   &         \phantom{i}    &         \phantom{i}   &\phantom{i}  
        \end{bNiceMatrix}
        \begin{bNiceMatrix}
            I_q \\ xI_q \\ yI_q \\ \Vdots[shorten-end=5pt]
        \end{bNiceMatrix}
        = \begin{bNiceMatrix}
            \mathscr{L}_{(0,0)} \\
            x\mathscr{L}_{(1,1)} + \mathscr{L}_{(1,0)} \\
            y\mathscr{L}_{(2,2)} + x\mathscr{L}_{(2,1)} + \mathscr{L}_{(2,0)} \\
            \Vdots[shorten-end=3pt]
        \end{bNiceMatrix}.    \end{align*}
Here $\mathscr{L}_{(I,K)}\in\R^{q\times q} $  and $\mathscr{L}_{(I,I)}$  are invertible and lower triangular. The leading matrix coeficient of $B_{(i,j)}(\boldsymbol{x})$ is $\mathscr{L}_{(I,I)}$, with $I \sim (i,j)$. Therefore,  the grlex-degree structure follows. A similar argument follows for $A(\boldsymbol x)$, but now with $   \mathscr{U}_{(I,K)}\in\R^{p\times p} $  and $   \mathscr{U}_{(I,I)}$  invertible and upper unitriangular.

    The \emph{normal} degree bounds for the individual entries were previously established in \cite{BTP} for the univariate case. The extension to two variables follows by identifying the \emph{normal} degrees in the univariate case with the corresponding graded lexicographic positions. Lastly, the sequence introduced in Definition \ref{Def: Sequence 1} determines the grlex-degree, which is the pair of numbers whose position in the sequence is given by $\lpos$.
\end{proof}
Let us introduce additional notation for the two families of orthogonal polynomials, consistent with Definition \ref{Def: Decomposition2}.

\begin{Definition}
We define the following decomposition for the two families of orthogonal polynomials:
\begin{align*}
    B_{[n]}(\boldsymbol{x}) & = \begin{bNiceMatrix}
        B_{(n,0)}(\boldsymbol{x}) \\ \Vdots \\ B_{(n,n)}(\boldsymbol{x})
    \end{bNiceMatrix}\in\R^{(n+1)q\times q}(\boldsymbol x), & 
    A_{[n]}(\boldsymbol{x}) & = \begin{bNiceMatrix}
        A_{(n,0)}(\boldsymbol{x}) & \Cdots & A_{(n,n)}(\boldsymbol{x})
    \end{bNiceMatrix}\in\R^{p\times (n+1)p}(\boldsymbol x).
\end{align*}
These are block matrices whose diagonal entries contain leading polynomial terms ranging from \( x^n \) in the first block to \( y^n \) in the last.
\end{Definition}
%Using the notation established in the proof of Proposition \ref{Prop: DegreeStructure}, we obtain:
We have
\[
    \begin{bNiceMatrix}
        B_{[0]} (\boldsymbol{x}) \\ 
        B_{[1]} (\boldsymbol{x}) \\ 
        B_{[2]} (\boldsymbol{x}) \\ 
        \Vdots[shorten-end=5pt]
    \end{bNiceMatrix} = 
\begin{bNiceMatrix}
	\mathscr{L}_{[0,0]} & 0 & 0 & 0 & \Cdots[shorten-end=-5pt] &\phantom{i}  \\
	\mathscr{L}_{[1,0]}&    \mathscr{L}_{[1,1]} & 0 & 0 & \Cdots[shorten-end=-5pt] &\phantom{i}  \\
	\mathscr{L}_{[2,0]}&    \mathscr{L}_{[2,1]}&    \mathscr{L}_{[2,2]} & 0 & \Cdots[shorten-end=-5pt] &\phantom{i}  \\
	\Vdots[shorten-end=-8pt] & \Vdots[shorten-end=-8pt]  & \Vdots[shorten-end=-8pt]  & \Ddots[shorten-end=-15pt]  & \Ddots[shorten-end=-5pt] &\phantom{i}  
	\\
	\phantom{i}   &         \phantom{i}   &         \phantom{i}   &         \phantom{i}    &         \phantom{i}   &\phantom{i}  
\end{bNiceMatrix}
    \begin{bNiceMatrix}
        I_q \\ xI_q \\ yI_q \\ x^2 I_q \\ xy I_q \\ y^2 I_q \\ \Vdots[shorten-end=5pt] 
    \end{bNiceMatrix} 
    = 
    \begin{bNiceMatrix}
        \mathscr{L}_{[0,0]} \\[4pt] 
        \mathscr{L}_{[1,1]}\left[\begin{array}{c}
        xI_q \\ yI_q 
        \end{array} \right] + \mathscr{L}_{[1,0]} \\[4pt]
        \mathscr{L}_{[2,2]}\left[\begin{array}{c}
        x^2I_q \\ xyI_q \\ y^2 I_q 
        \end{array} \right] + \mathscr{L}_{[2,1]}\left[\begin{array}{c}
        xI_q \\ yI_q 
        \end{array} \right] + \mathscr{L}_{[2,0]} \\ 
        \Vdots[shorten=70pt][shorten-end=3pt] 
    \end{bNiceMatrix},
\]  
where $\mathscr L_{[i,j]}\in\R^{(i+1)q\times (j+1)q}$, and  $\mathscr L_{[i,i]}\in\R^{(i+1)q\times (i+1)q}$ are nonsingular lower triangular matrices, which gives us a block characterization of $B_{[n]}(\boldsymbol{x})$. Equivalently, we would find a similar characterization of $A_{[n]}(\boldsymbol{x})$. We now establish the orthogonality relations satisfied by the matrix polynomial families \( A(\boldsymbol{x}) \) and \( B(\boldsymbol{x}) \).

After all this efforts, many of notational nature given the intricacies of the bivariate multiple of the mixed type nature of the moment matrix, we are able to give both the bivariate multiple orthogonal of mixed type relations satisfied by the polynomials, as well as the corresponding biorthogonal relations they satisfy.

\begin{Proposition}[Orthogonality relations]
The following orthogonality relations hold:
\begin{align*}
    \sum_{a=1}^p \int_\Delta x_1^{k-l}x_2^{l} \d \mu_{b,a}(\boldsymbol{x}) A_{n}^{(a)}(\boldsymbol{x}) & = 0, & 
    K \in \{ 0, \dots , \left\lceil \frac{n+1-b}{q} \right\rceil - 1\}, \quad b \in \{ 1,\dots,q \}, \\
    \sum_{b=1}^q \int_\Delta B_{n}^{(b)}(\boldsymbol{x}) \d \mu_{b,a}(\boldsymbol{x}) x_1^{k-l}x_2^{l} & = 0, & 
    K \in \{ 0, \dots , \left\lceil \frac{n+1-a}{p} \right\rceil - 1\}, \quad a \in \{ 1,\dots,p \},
\end{align*}
where in both cases, \( K \sim (k,l) \). The maximun $K$ matches the grlex-position of the upper bound for $B_{n-1}^{(b)}(\boldsymbol{x})$ and that for $A_{n-1}^{(a)}(\boldsymbol{x})$.
\end{Proposition}

\begin{proof}
These relations follow directly from the Gauss–Borel factorization of the moment matrix and Definition \ref{Def: OrthogonalPoly}. On the one hand, we have:
\begin{equation*}
    \int_{\Delta} X_{[q]}(\boldsymbol{x}) \d \mu (\boldsymbol{x}) A(\boldsymbol{x}) = S^{-1}H,
\end{equation*}
where \( S^{-1}H \) is a lower triangular matrix. Considering the entries above the main diagonal yields the first orthogonality relation.

Similarly, the second relation follows from:
\begin{equation*}
    \int_{\Delta} B(\boldsymbol{x}) \d \mu (\boldsymbol{x}) X_{[p]}^\top(\boldsymbol{x}) = \bar{S}^{-\top},
\end{equation*}
with \( \bar{S}^{-\top} \) being upper unitriangular.
\end{proof}

\begin{Proposition}[Biorthogonality relations]
The following biorthogonality relation is satisfied:
\begin{equation*}
    \int_{\Delta} B(\boldsymbol{x}) \d \mu (\boldsymbol{x}) A(\boldsymbol{x}) = I.
\end{equation*}
In component form, this can be expressed as:
\begin{equation*}
    \sum_{a=1}^p \sum_{b = 1}^q \int_{\Delta} B_m^{(b)}(\boldsymbol{x}) \d \mu_{b,a}(\boldsymbol{x}) A_{n}^{(a)}(\boldsymbol{x}) = \delta_{n,m},
\end{equation*}
where \( \delta_{n,m} \) denotes the Kronecker delta.
\end{Proposition}
\section{Recurrence Relations}
We now study the recurrence matrices and the recurrence relations satisfied by the corresponding families of orthogonal polynomials. Several results will be presented in the general case through illustrative examples, as the underlying arguments remain essentially the same.
\begin{Proposition}[Hankel type symmetry] \label{Prop: MSym-GeneralCase}
	The moment matrix satisfies:
	\begin{equation*}
		\Lambda_{[q];k} \mathscr{M} = \mathscr{M}\Lambda_{[p];k}^\top,
	\end{equation*}
	for all $k \in \{1,2\}$.
\end{Proposition}
\begin{proof}
This result follows from:
\[\begin{aligned}
	\Lambda_{[q],k}\mathscr{M} &= \int_\Delta \Lambda_{[q],k}X_{[q]}(\boldsymbol{x}) \, \mathrm{d} \mu (\boldsymbol{x}) X_{[p]}^\top(\boldsymbol{x}) \\&= \int_\Delta x_k X_{[q]}(\boldsymbol{x}) \, \mathrm{d} \mu (\boldsymbol{x}) X_{[p]}^\top(\boldsymbol{x}) \\
	&= \int_\Delta X_{[q]}(\boldsymbol{x}) \, \mathrm{d} \mu (\boldsymbol{x}) X_{[p]}^\top(\boldsymbol{x}) \Lambda_{[p],k}^\top \\&= \mathscr{M}\Lambda_{[p],k}^\top.
\end{aligned}\]
\end{proof}
\begin{Definition} \label{Def: TMatrix-GeneralCase}
	We define the recurrence matrices as
	\begin{equation*}
		T_k \coloneq S \Lambda_{[q];k} S^{-1},
	\end{equation*}
	for each $k \in \{1,2\}$.
\end{Definition}

\begin{Proposition} \label{Prop: TMatrix-StudyCase}
	The recurrence matrix satisfies the following dual relation:
	\begin{equation*}
		T_k = S \Lambda_{[q];k} S^{-1} = H \bar{S}^{-\top} \Lambda_{[p];k}^\top \bar{S}^\top H^{-1}.
	\end{equation*}
\end{Proposition}
\begin{proof}
    It is a direct consequence of Proposition \ref{Prop: MSym-GeneralCase} and Definition \ref{Def: TMatrix-GeneralCase}.
\end{proof}
Let us study the structural properties of $T_k$ that arise from this dual relation. These will be presented using an example and later generalized. 

\textbf{Example:}
    Let us take $q=1$ and $p=2$ and focus on the case $k=1$. Since both $S$ and $S^{-1}$ are unitriangular, and the matrix $\Lambda_{[1];1}$ contains ones above the main diagonal (located in specified rows and distinct columns as given in Proposition \ref{Prop: Position of 1s Lambda} for $r = 1$), the product $S\Lambda_{[1];1}S^{-1}$ preserves the lower unitriangular structure but shifts it upward along the positions determined by the $1$s in $\Lambda_{[1];1}$. Conceptually, each $1$ in $\Lambda_{[1];1}$ casts a “shadow” to the left in the resulting matrix, consisting of constant (possibly zero) entries. Therefore, we deduce:
    \[
        T_1 = S\Lambda_{[1];1}S^{-1} = \begin{bNiceMatrix}
	   \stackrel{\leftarrow}{\ast} & 1 & 0 & 0 & 0 & 0 & 0 & \Cdots[shorten-end=5pt]  \\
	   \ast & \ast & \stackrel{\leftarrow}{\ast} & 1 & 0 & 0 & 0 & \Cdots[shorten-end=5pt]  \\
	   \ast & \ast & \ast & \stackrel{\leftarrow}{\ast} & 1 & 0 & 0 & \Cdots[shorten-end=5pt]  \\
	   \ast & \ast & \ast & \ast & \ast & \stackrel{\leftarrow}{\ast} & 1 &  \\
	   \ast & \ast & \ast & \ast & \ast & \ast & \ast &   \\
	   \Vdots[shorten-end=2pt] & \Vdots[shorten-end=2pt]  & \Vdots[shorten-end=2pt]  & \Vdots[shorten-end=2pt]  & \Vdots[shorten-end=2pt]  & \Vdots[shorten-end=2pt]  & \Vdots[shorten-end=2pt]  & 
        \end{bNiceMatrix},
    \]
    where $\ast$ denotes arbitrary real entries and $\stackrel{\leftarrow}{\ast}$ denotes the projected influence of a $1$-entry.

    From the second expression for $T_1$,
    \[
        T_1 = \left( H^{-1}\bar{S}\Lambda_{[2];1}\bar{S}^{-1}H \right)^\top = H \left( \bar{S}\Lambda_{[2];1}\bar{S}^{-1} \right)^\top H^{-1},
    \]
    we observe that a similar structure arises in the matrix $\bar{S}\Lambda_{[2];1}\bar{S}^{-1}$. The matrices $H$ and $H^{-1}$ transform the unit entry $(n,m)$ by multiplication by the factor \[\frac{H_n}{H_m}. \]Finally, the transpose operation completes the transformation. Thus,
    \[
	T_1 = H \bar{S}^{-\top} \Lambda_{[2];1}^\top \bar{S}^\top H^{-1} = \begin{bNiceMatrix}
		* & * & * & * & * & \Cdots[shorten-end=5pt]  \\
		\astuparrow & \ast & * & * & * & \Cdots[shorten-end=5pt]  \\
		\boxed{\cdot} & \astuparrow & * & * & * & \Cdots[shorten-end=5pt]  \\
		0 & \boxed{\cdot} & * & * & * & \Cdots[shorten-end=5pt]  \\
		0 & 0 & \ast & * & *& \Cdots[shorten-end=5pt]  \\
		0 & 0 & \astuparrow & * & *& \Cdots[shorten-end=5pt]  \\
		0 & 0 & \boxed{\cdot} & \astuparrow & *& \Cdots[shorten-end=5pt]  \\
		0 & 0 & 0 & \boxed{\cdot} & *& \Cdots[shorten-end=5pt]  \\
		\Vdots[shorten-end=2pt] & \Vdots[shorten-end=2pt] & \Vdots[shorten-end=2pt] &  &  &  
	\end{bNiceMatrix},
    \]
    where $\boxed{\cdot}$ denotes non-zero entries derived from the original unit entries. By comparing the two relations, we derive the explicit structure of $T_1$. Note that when a single projection shadow intersects either a 1 or an $\boxed{\cdot}$ entry, it vanishes. Similarly, the intersection of two projection shadows results in their cancellation.

    Finally, we present the  form of $T_1$:
	\[
		T_1=\begin{bNiceMatrix}
		    \boldsymbol  \ast & 1 & 0 & 0 & 0 & 0 & 0 & 0 & 0 & 0 & 0 & 0 & 0 & 0 &  \Cdots[shorten-end=-5pt]&&&&\phantom{i}  \\
		      \ast & \boldsymbol\ast & \ast & 1 & 0 & 0 & 0 & 0 & 0 & 0 & 0 & 0 & 0 & 0 &  \Cdots[shorten-end=-5pt]&&&&\phantom{i}  \\
		      \boxed{\cdot} & \ast & \boldsymbol\ast & \ast & 1 & 0 & 0 & 0 & 0 & 0 & 0 & 0 & 0 & 0 &  \Cdots[shorten-end=-5pt]  &&&&\phantom{i}\\
		      0 & \boxed{\cdot} & \ast &\boldsymbol \ast & \ast & \ast & 1 & 0 & 0 & 0 & 0 & 0 & 0 & 0 &  \Cdots[shorten-end=-5pt] &&&&\phantom{i} \\
		      0 & 0 & \ast & \ast &\boldsymbol \ast & \ast & \ast & 1 & 0 & 0 & 0 & 0 & 0 & 0 &  \Cdots[shorten-end=-5pt]&&&&\phantom{i}   \\
		      0 & 0 & \ast & \ast & \ast & \boldsymbol\ast & \ast & \ast & 1 & 0 & 0 & 0 & 0 & 0 &  \Cdots[shorten-end=-5pt]  &&&&\phantom{i} \\
		      0 & 0 & \boxed{\cdot} & \ast & \ast & \ast &\boldsymbol \ast & \ast & \ast & \ast & 1 & 0 & 0 & 0 &  \Cdots[shorten-end=-4pt] &&&&\phantom{i} \\
		      0 & 0 & 0 & \boxed{\cdot} & \ast & \ast & \ast & \boldsymbol\ast & \ast & \ast & \ast & 1 & 0 & 0 &  \Cdots[shorten-end=-2pt]&&&&\phantom{i}\\
		      0 & 0 & 0 & 0 & \boxed{\cdot} & \ast & \ast & \ast & \boldsymbol\ast & \ast & \ast & \ast & 1 & 0 &  \Cdots[shorten-end=-2pt] &&&&\phantom{i} \\
		      0 & 0 & 0 & 0 & 0 & \boxed{\cdot} & \ast & \ast & \ast & \boldsymbol\ast & \ast & \ast & \ast & 1 &  &\phantom{i} \\
		      0 & 0 & 0 & 0 & 0 & 0 & \ast & \ast & \ast & \ast & \boldsymbol\ast & \ast & \ast & \ast &   &&&&\phantom{i}  \\[-5pt]
		      0 & 0 & 0 & 0 & 0 & 0 & \ast & \ast & \ast & \ast & \ast &\boldsymbol \ast & \ast & \ast &  \Ddots[shorten-end=-5pt]  &&&&\phantom{i}  \\[-5pt]
		      0 & 0 & 0 & 0 & 0 & 0 & \boxed{\cdot} & \ast & \ast & \ast & \ast & \ast & \boldsymbol\ast & \ast &  \Ddots[shorten-end=-5pt]&&&&\phantom{i}  \\[-2pt]
		      0 & 0 & 0 & 0 & 0 & 0 & 0 & \boxed{\cdot} & \ast & \ast & \ast & \ast & \ast & \boldsymbol\ast &  \Ddots[shorten-end=-5pt]  &&&&\phantom{i}\\
		      0 & 0 & 0 & 0 & 0 & 0 & 0 & 0 & \boxed{\cdot} & \ast & \ast & \ast & \ast & \ast &  \Ddots[shorten-end=-5pt] &&&&\phantom{i}\\
		      0 & 0 & 0 & 0 & 0 & 0 & 0 & 0 & 0 & \boxed{\cdot} & \ast & \ast & \ast & \ast & \Ddots[shorten-end=0pt] &&&&\phantom{i} \\
		      0 & 0 & 0 & 0 & 0 & 0 & 0 & 0 & 0 & 0 & \boxed{\cdot} & \ast & \ast & \ast & \Ddots[shorten-end=10pt] &&&&\phantom{i} \\
		%0 & 0 & 0 & 0 & 0 & 0 & 0 & 0 & 0 & 0 & 0 & \boxed{d} & \ast & \ast & \Cdots[shorten-end=5pt]  \\
		      \Vdots[shorten-end=-6pt] & \Vdots[shorten-end=-6pt]  & \Vdots[shorten-end=-6pt]  & \Vdots[shorten-end=-6pt]  & \Vdots[shorten-end=-6pt]  & \Vdots[shorten-end=-6pt]  & \Vdots[shorten-end=-6pt]  & \Vdots[shorten-end=-6pt]  & \Vdots[shorten-end=-6pt]  & \Vdots[shorten-end=-6pt]  &  &  & \Ddots[shorten-end=14pt]  & \Ddots[shorten-end=18pt]  & \Ddots[shorten-end=14pt] &&&&\phantom{i}\\
		      \phantom{i}&\phantom{i}&\phantom{i}&\phantom{i}&\phantom{i}&\phantom{i}&\phantom{i}&\phantom{i}&\phantom{i}&\phantom{i}&\phantom{i}&\phantom{i}&\phantom{i}&\phantom{i}&\phantom{i}&&&&\phantom{i}
	    \end{bNiceMatrix}.
	\]     \enlargethispage{1cm}
    Similar arguments lead to the following recurrence matrix for $k=2$,
	\[
		T_2=\begin{bNiceMatrix}
		\boldsymbol \ast & 1 & 0 & 0 & 0 & 0 & 0 & 0 & 0 & 0 & 0 & 0 & 0 &0& \Cdots[shorten-end=-3pt]&&&\phantom{i}   \\
		\ast & \boldsymbol\ast & \ast & \ast & 1 & 0 & 0 & 0 & 0 & 0 & 0 & 0 & 0 & 0 & \Cdots[shorten-end=-3pt]&&&\phantom{i}  \\
		\ast & \ast &\boldsymbol \ast & \ast & \ast & 1 & 0 & 0 & 0 & 0 & 0 & 0 & 0 & 0 &  \Cdots[shorten-end=-3pt]&&&\phantom{i} \\
		\boxed{\cdot} & \ast & \ast & \boldsymbol\ast & \ast & \ast & \ast & 1 & 0 & 0 & 0 & 0 & 0 & 0 & \Cdots[shorten-end=-3pt]&&&\phantom{i}  \\
		0 & \boxed{\cdot} & \ast & \ast & \boldsymbol\ast & \ast & \ast & \ast & 1 & 0 & 0 & 0 & 0 & 0 & \Cdots[shorten-end=-3pt]&&&\phantom{i}  \\
		0 & 0 & \ast & \ast & \ast & \boldsymbol\ast & \ast & \ast & \ast & 1 & 0 & 0 & 0 & 0 &  \Cdots[shorten-end=-3pt]&&&\phantom{i} \\
		0 & 0 & \ast & \ast & \ast & \ast & \boldsymbol\ast & \ast & \ast & \ast & \ast & 1 & 0 & 0 & \Cdots[shorten-end=-3pt]&&&\phantom{i}  \\
		0 & 0 & \boxed{\cdot} & \ast & \ast & \ast & \ast &\boldsymbol \ast & \ast & \ast & \ast & \ast & 1 & 0 & \Cdots[shorten-end=-3pt]&&&\phantom{i} \\
		0 & 0 & 0 & \boxed{\cdot} & \ast & \ast & \ast & \ast & \boldsymbol\ast & \ast & \ast & \ast & \ast & 1 &   \\
		0 & 0 & 0 & 0 & \boxed{\cdot} & \ast & \ast & \ast & \ast & \boldsymbol\ast & \ast & \ast & \ast & \ast &  &&\phantom{i}  \\[-5pt]
		0 & 0 & 0 & 0 & 0 & \boxed{\cdot} & \ast & \ast & \ast & \ast &\boldsymbol \ast & \ast & \ast & \ast &  \Ddots[shorten-end=-20pt]&&&\phantom{i}  \\[-5pt]
		0 & 0 & 0 & 0 & 0 & 0 & \ast & \ast & \ast & \ast & \ast & \boldsymbol\ast & \ast & \ast &  \Ddots[shorten-end=-3pt]&&&\phantom{i}   \\[-5pt]
		0 & 0 & 0 & 0 & 0 & 0 & \ast & \ast & \ast & \ast & \ast & \ast &\boldsymbol \ast & \ast &  \Ddots[shorten-end=-3pt] &&\phantom{i}  \\[-5pt]
		0 & 0 & 0 & 0 & 0 & 0 & \boxed{\cdot} & \ast & \ast & \ast & \ast & \ast & \ast & \boldsymbol\ast & \Ddots[shorten-end=-20pt]  &&&\phantom{i} \\
		0 & 0 & 0 & 0 & 0 & 0 & 0 & \boxed{\cdot} & \ast & \ast & \ast & \ast & \ast & \ast &  \Ddots[shorten-end=-3pt]  \\
		0 & 0 & 0 & 0 & 0 & 0 & 0 & 0 & \boxed{\cdot} & \ast & \ast & \ast & \ast & \ast &  \Ddots[shorten-end=-0pt] &&\phantom{i}  \\
		0 & 0 & 0 & 0 & 0 & 0 & 0 & 0 & 0 & \boxed{\cdot} & \ast & \ast & \ast & \ast & \Ddots[shorten-end=0pt]  &&&\phantom{i} \\
		0 & 0 & 0 & 0 & 0 & 0 & 0 & 0 & 0 & 0 & \boxed{\cdot} & \ast & \ast & \ast & \Ddots[shorten-end=-20pt]  &&&\phantom{i} \\
		%0 & 0 & 0 & 0 & 0 & 0 & 0 & 0 & 0 & 0 & 0 & \boxed{d} & \ast & \ast & \Cdots[shorten-end=5pt]  \\
	\\
	%0 & 0 & 0 & 0 & 0 & 0 & 0 & 0 & 0 & 0 & 0 & \boxed{d} & \ast & \ast & \Cdots[shorten-end=5pt]  \\
	\Vdots[shorten-end=-10pt] & \Vdots[shorten-end=-10pt]  & \Vdots[shorten-end=-10pt]  & \Vdots[shorten-end=-10pt]  & \Vdots[shorten-end=-10pt]  & \Vdots[shorten-end=-10pt]  & \Vdots[shorten-end=-10pt]  & \Vdots[shorten-end=-10pt]  & \Vdots[shorten-end=-10pt]  & \Vdots[shorten-end=-10pt]  &  &  &   & \Ddots[shorten-end=-30pt]  & \Ddots[shorten-end=-10pt] &\Ddots[shorten-end=-15pt]&\phantom{i}\\
	\phantom{i}&\phantom{i}&\phantom{i}&\phantom{i}&\phantom{i}&\phantom{i}&\phantom{i}&\phantom{i}&\phantom{i}&\phantom{i}&\phantom{i}&\phantom{i}&\phantom{i}&\phantom{i}&\phantom{i}&&&\phantom{i}
	\end{bNiceMatrix}.
	\] 
	Note that for reference we have marked as a bold ast $\boldsymbol{\ast}$ the diagonal entries.
	
	Let us characterize all the nonzero (or potentially nonzero) entries of $T_k$, which we denote as $T_{k;i,j}$, for a given row.
    We have the following structure for the $n$-th row of $T_k$: the position of the 1, which is the last nonzero entry, is determined by the position of the $1$s in $\Lambda_{[1],k}$, that is, $(n,n^+_{1;k})$, where $n^+_{1;k} = n+\mathcal{F}\left( n \right)+k$.
    Now, consider the position of the first entry in the $n$-th row. For example:
    \[
        T_1 = \begin{bNiceMatrix}
            \ast & 1 & 0 & 0 & \Cdots[shorten-end=5pt]  \\
            \ast & \ast & \ast & 1 & \Cdots[shorten-end=5pt]  \\
            \boxed{\cdot} & \ast & \ast & \ast & \Cdots[shorten-end=5pt]  \\
            0 & \boxed{\cdot} & \ast & \ast & \Cdots[shorten-end=5pt]  \\
            0 & 0 & \ast & \ast & \Cdots[shorten-end=5pt]  \\
            0 & 0 & \ast & \ast & \Cdots[shorten-end=5pt]  \\
            0 & 0 & \boxed{\cdot} & \ast & \Cdots[shorten-end=5pt]  \\
            0 & 0 & 0 & \boxed{\cdot} & \Cdots[shorten-end=5pt] \\
            \Vdots[shorten-end=2pt] & \Vdots[shorten-end=2pt] & \Vdots[shorten-end=2pt] & \Vdots[shorten-end=2pt]
        \end{bNiceMatrix},
    \]
    Rows 4 and 5 (counting from zero) exhibit a different structure from row 6. For row 6, and generally for those rows where $n \in \mathbb{N}_0 \setminus J_{[2],1}$, the mapping $n_2^- = n - 2\mathcal{F}^-_k\left( \frac{n}{2} \right)$ is bijective, and the pair $(n,n^-_2)$ identifies the position of the nonzero element.
    For rows 4 and 5, the structure differs, as they do not end in a strictly nonzero entry. In the general case, if the $n$-th row exhibits this behavior, then $n \in J_{[2],1}$. In such cases, the first column of possibly nonzero entries is determined by the next row, denoted by $N$, that satisfies $N \in \mathbb{N}_0 \setminus J_{[2],1}$.

    Lastly, the strictly non-zero elements (denoted earlier as $\boxed{\cdot}$) can be computed as $T_{k;\,n,m} = \frac{H_n}{H_m}$, as already mentioned.

	To characterize the first potentially non-zero element in a given row for the general case, we introduce the following quantity.

\begin{Definition}
    For a given $n \in \N_0$, we define the associated number $N^-_{r;k}$ as follows:
    \begin{itemize}
		\item If $n \in \mathbb{N}_0 \setminus J_{r;k}$, then $N^-_{r;k} \coloneq n - r \mathcal{F}^-_k\left( \frac{n}{r} \right)$.
		\item If $n \notin \mathbb{N}_0 \setminus J_{r;k}$, define $N$ as the smallest integer satisfying $n < N$ and $N \in \mathbb{N}_0 \setminus J_{r;k}$. Then,
		\[
		    N^-_{r;k} \coloneq N - r \mathcal{F}^-_k\left( \frac{N}{r} \right).
		\]
	\end{itemize}
\end{Definition}

Building on the preceding definition and illustrative example, we characterize an arbitrary row of $T_k$ as follows.

\begin{Proposition}\label{Prop: Posiciones T_1 rows-GeneralCase}
    The structure of the $n$-th row of $T_k$ is given by:
    \begin{itemize}
        \item The position of the entry equal to $1$, which is the last nonzero entry, corresponds to the index pair $(n,n^+_{q;k})$, where $n^+_{q;k} = n+q\mathcal{F}_k^{+}\left( \frac{n}{q} \right)$.
        \item The position of the first potentially nonzero entry is given by the pair $(n, N_{p;k}^-)$.
    \end{itemize}
    In summary, the $n$-th row may contain nonzero entries ranging from $(n, N_{p;k}^-)$ to $(n,n_{q;k}^+)$. The entry at $(n, N_{p;k}^-)$ is strictly nonzero whenever $n \in \mathbb{N}_0 \setminus J_{p;k}$, in which case it satisfies
    \[
        T_{k;\,n,N_{p;k}^-} = \frac{H_n}{H_{N_{p;k}^-}}.
    \]
\end{Proposition}

Similar arguments lead to the following description of the columns of $T_k$. 

\begin{Proposition} \label{Prop: Posiciones T_1 columns-GeneralCase}
    The structure of the $n$-th column of $T_k$ is described as follows:
    \begin{itemize}
        \item The position of the last strictly nonzero entry is given by $(n^+_{p;k},n)$, where $n^+_{p;k} = n+p\mathcal{F}_k^{+}\left( \frac{n}{p} \right)$.
        \item The position of the first potentially nonzero entry is given by $( N^-_{q;k},n )$.
    \end{itemize}
    Thus, the $n$-th column may contain nonzero entries between $(N_{q;k}^-,n)$ and $(n_{p;k}^+,n)$. The first entry, $(N_{q;k}^-,n)$, is equal to $1$ whenever $n \in \mathbb{N}_0 \setminus J_{q;k}$. Moreover, the final strictly nonzero entry satisfies
    \[
        T_{k;\,n_{p;k}^+,n} = \frac{H_{n_{p;k}^+}}{H_n}.
    \]
\end{Proposition}

With these characterizations of the matrices of the recurrence we are able to state the following recurrence relations.

\begin{Theorem}[Recurrence relations] \label{Recurrence Relation-GeneralCase}
For $k\in\{1,2\}$, the families of bivariated multiple orthogonal polynomials of the mixed type satisfy the following recurrence relations:
\[	\begin{aligned}
		T_k B(\boldsymbol{x}) & = x_k B(\boldsymbol{x}), &
		A(\boldsymbol{x}) T_k &= x_k A(\boldsymbol{x}).
	\end{aligned}\]
	Entrywise, these relations take the form:
\[	\begin{aligned}
		x_k B^{(b)}_n(\boldsymbol{x}) &= B^{(b)}_{n_{q;k}^+}(\boldsymbol{x}) + \sum_{i=N_{p;k}^-}^{n_{q;k}^+-1} T_{k;\,n,i}B^{(b)}_i(\boldsymbol{x}), & b &\in \{1, \dots, q\}, &k\in\{1,2\},\\
		x_k A_n^{(a)}(\boldsymbol{x}) &= T_{k;\,n_{p;k}^+,n} A^{(a)}_{n_{p;k}^+}(\boldsymbol{x}) + \sum_{i=N_{q;k}^-}^{n_{p;k}^+-1} T_{k;\,i,n}A^{(a)}_i(\boldsymbol{x}),& a &\in \{1, \dots, p\},&k\in\{1,2\}.
	\end{aligned}\]
\end{Theorem}

\begin{proof}
We prove the case for $A(\boldsymbol{x})$, as the case for $B(\boldsymbol{x})$ follows analogously. Since $X_{[p]}(\boldsymbol{x})^\top\Lambda_{[p],k}^\top = x_k X_{[p]}(\boldsymbol{x})^\top$, we obtain:
\[
    A(\boldsymbol{x}) T_k = A(\boldsymbol{x}) H \bar{S}^{-\top} \Lambda_{[p];k}^\top \bar{S}^\top H^{-1} = X^\top_{[p],k} \Lambda_{[p];k}^\top \bar{S}^\top H^{-1} = x_k A(\boldsymbol{x}).
\]
The entrywise expressions follow directly from this identity and Propositions \ref{Prop: Posiciones T_1 rows-GeneralCase} and \ref{Prop: Posiciones T_1 columns-GeneralCase}.
\end{proof}
We can also establish matrix-wise recurrence relations for the orthogonal polynomials. These relations reveal a structural similarity between standard bivariate orthogonal polynomials with a scalar weight and bivariate multiple orthogonal polynomials of mixed type. The underlying idea is as follows: one of the relations, $A(x)T_k = x_k A(x)$ or $T_kB(x) = x_k B(x)$, can be expressed as a three-term recurrence relation with matrix coefficients (although, as we will see, in certain cases four terms are required). The other relation can equivalently be formulated as an $\alpha$-term recurrence relation, where $\alpha$ is a constant that serves as an upper bound for the number of terms involved. To proceed, let us examine some properties of the function $\mathcal{F}(x)$.

\begin{Proposition} \label{Prop: Properties F}
    The function $\mathcal{F}(x)$ introduced in Definition \ref{Def: Special Function} possesses the following properties. Let $x \in \R$:
    \begin{itemize}
        \item $\mathcal{F}(x)$ takes values in $\N_0$, and it is constant on each interval $\left[ \frac{1}{2}i(i+1), \frac{1}{2}(i+1)(i+2) \right)$.
        \item $\mathcal{F}\left( \left[ \frac{1}{2}i(i+1), \frac{1}{2}(i+1)(i+2) \right) \right) = i$.
    \end{itemize}
\end{Proposition}
\begin{proof}
    This follows directly from the definition of $\mathcal{F}(x)$ and its connection to the graded-lexicographic sequence introduced in Definition \ref{Def: Sequence 1}.
\end{proof}

Throughout the remainder of this section, we assume that $q \leq p$. The case $p < q$ leads to analogous results upon interchanging $A(x) \leftrightarrow B(x)$ and $p \leftrightarrow q$.

\begin{Proposition}[Recurrence Relations for the $A(x)$ Family] \label{Prop: RR for A - Matrix}
    The relation 
    \[
        A(x) T_k = x_k A(x),
    \]
    with $k \in \{1,2\}$, can be written as follows:
    \begin{itemize}
        \item For $k = 1$ and $n \in \N_0$: 
        \[
            x_1 A_{[n]}(\boldsymbol{x}) = a_{[n+1,n];1}A_{[n+1]}(\boldsymbol{x}) + a_{[n,n];1}A_{[n]}(\boldsymbol{x}) + a_{[n-1,n];1}A_{[n-1]}(\boldsymbol{x}),
        \]
        where $a_{[i,j];1} \in \R^{(i+1)p\times(j+1)p}$ and 
        \[
        a_{[n+1,n];1} = \begin{bNiceMatrix} 
        \boxed{\cdot} & \ast & \Cdots & \ast \\
        0 & \Ddots & \Ddots & \Vdots \\
        \Vdots & \Ddots & & \ast \\
        0 & \Cdots & 0 & \boxed{\cdot} \\
        0 & \Cdots & & 0 \\
        \Vdots & & & \Vdots \\
        0 & \Cdots & & 0
        \CodeAfter
        \tikz \draw [solid] (1-|5) -- (5-|5);
        \tikz \node [right=1pt] at (3-|5) {$(n+1)p$};
        \tikz \draw [solid] (5-|5) -- (5-|1);
        \tikz \node [left=1pt] at (6.5-|1) {$p$};
        \tikz \draw [solid] (5-|1) -- (8-|1);
    \end{bNiceMatrix}. 
        \]
        \item For $k = 2$ and $n \in \N_0$, two possibilities arise: \\
        If $\frac{p}{q} \in \{1\} \cup \left[ \frac{3}{2}, \infty \right)$, we obtain a three-term recurrence relation: 
        \[
            x_2 A_{[n]}(\boldsymbol{x}) = a_{[n+1,n];2}A_{[n+1]}(\boldsymbol{x}) + a_{[n,n];2}A_{[n]}(\boldsymbol{x}) + a_{[n-1,n];2}A_{[n-1]}(\boldsymbol{x}).
        \]
        If $1 < \frac{p}{q} < \frac{3}{2}$, a four-term recurrence relation holds: 
        \[
            x_2 A_{[n]}(\boldsymbol{x}) = a_{[n+1,n];2}A_{[n+1]}(\boldsymbol{x}) + a_{[n,n];2}A_{[n]}(\boldsymbol{x}) + a_{[n-1,n];2}A_{[n-1]}(\boldsymbol{x}) + a_{[n-2,n];2}A_{[n-2]}(\boldsymbol{x}),
        \]
        where $a_{[i,j];2} \in \R^{(i+1)p\times(j+1)p}$ and 
        \[
        a_{[n+1,n];2} = \begin{bNiceMatrix} 
        \ast & \Cdots & & \ast \\
        \Vdots & & & \Vdots \\
        \ast & \Cdots & & \ast \\
        \boxed{\cdot} & \ast & \Cdots & \ast \\
        0 & \Ddots & \Ddots & \Vdots \\
        \Vdots & \Ddots & & \ast \\
        0 & \Cdots & 0 & \boxed{\cdot} \\
        \CodeAfter
        \tikz \draw [solid] (1-|5) -- (4-|5);
        \tikz \node [right=1pt] at (3-|5) {$p$};
        \tikz \draw [solid] (4-|5) -- (4-|1);
        \tikz \node [left=1pt] at (6.5-|1) {$(n+1)p$};
        \tikz \draw [solid] (4-|1) -- (8-|1);
    \end{bNiceMatrix}, 
        \]
        in both cases.
    \end{itemize}
    The notation is such that $a_{[i,j];k} = 0$ whenever $i < 0$.
\end{Proposition}
\begin{proof}
    We show that the recurrence matrices $T_k$ can be expressed as 
    \[
        T_k = \begin{bNiceMatrix}
            a_{[0,0];k} & a_{[0,1];k} & 0_{[0,2]} & \Cdots[shorten-end=7pt] \\
            a_{[1,0];k} & a_{[1,1];k} & a_{[1,2];k} & 0_{[1,3]} & \Cdots[shorten-end=7pt] \\
            0_{[2,0]} & a_{[2,1];k} & \Ddots[shorten-end=-15pt] & \Ddots[shorten-end=-7pt] & \Ddots \\
            \Vdots & \Ddots[shorten-end=7pt] & \Ddots[shorten-end=7pt] & \phantom{i} & \phantom{i} & \phantom{i}
        \end{bNiceMatrix},
    \]
    and, for $k = 2 $ and $1 < \frac{p}{q} < \frac{3}{2}$, as 
    \[
        T_2 = \begin{bNiceMatrix}
            a_{[0,0];2} & a_{[0,1];2} & a_{[0,2];2} & 0_{[0,3]} & \Cdots[shorten-end=7pt] \\
            a_{[1,0];k} & a_{[1,1];k} & a_{[1,2];k} & a_{[1,3];k} & 0_{[1,4]} & \Cdots[shorten-end=7pt] \\
            0_{[2,0]} & a_{[2,1];k} & \Ddots[shorten-end=-15pt] & \Ddots[shorten-end=-15pt] & \Ddots[shorten-end=-10pt] & \Ddots \\
            \Vdots & \Ddots[shorten-end=7pt] & \Ddots[shorten-end=7pt] & \phantom{i} & \phantom{i} & \phantom{i} & \phantom{i}
        \end{bNiceMatrix}.
    \]
    The block structure along and below the main diagonal follows directly from Proposition \ref{Prop: TMatrix-StudyCase} (see the example provided in this section for further details). The structure of the blocks above the main diagonal requires a more technical argument.

    Consider the case $k = 1$. For a given $n \in \N$ and by Proposition \ref{Prop: Equiv Decom}, the submatrix
    \[
        \begin{bNiceMatrix}
            0_{[n-1,0]} & \Cdots & 0_{[n-1,n-3]} & a_{[n-1,n-2];1} & a_{[n-1,n-1];1} & a_{[n-1,n];1}
        \end{bNiceMatrix},
    \]
    contains the scalar entries of the original matrix $T_1$ ranging over columns $\{ 0, \dots, \frac{p}{2}n(n+3) + p -1\}$ and rows $\{ \frac{p}{2}n(n-1), \dots, \frac{p}{2}n(n+1) - 1 \}$. For a given row $m$, the position of the 1 is determined by the pair $(m,m_{q;1}^+)$, where 
    \[
        m_{q;1}^+ = m + q + q \mathcal{F}\left( \frac{m}{q} \right). 
    \]
    Therefore, we must show that
    \[
        \frac{p}{2}n(n+3) + p -1 \geq m_i + q + q\mathcal{F}\left( \frac{m_i}{q} \right),
    \]
    where $m_i = \frac{p}{2}n(n-1) + i$ and $i \in \{ 0, \dots, np - 1\}$. In simpler terms, for all rows of the aforementioned submatrix, the last column index is always greater than or equal to the column index at which the 1s in $T_1$ occur.

    To verify this inequality for all rows, it suffices to consider $m_i = \frac{p}{2}n(n+1) - 1$. Observe that
    \[
        \mathcal{F}\left(\frac{1}{2}\frac{p}{q}n(n+1) - \frac{1}{q}\right) \leq -\frac{1}{2} + \frac{1}{2} \sqrt{1+4\frac{p}{q}n(n+1)- \frac{8}{q}}.
    \]
    Consequently, a stronger inequality is given by
    \[
        2(n+1)\frac{p}{q} - 1 \geq \sqrt{1+4\frac{p}{q}n(n+1)}.
    \]
    Through straightforward algebraic manipulations, one finds that this holds whenever $\frac{p}{q} \geq 1$, which is already assumed.

    The case $k = 2$ can be analyzed analogously. However, we only state the difference that arises when $p \geq \frac{3}{2}$. For $n = 1$, the inequality corresponding to the last row becomes 
    \[
        \frac{p}{q} \geq 1 + \frac{1}{2} \mathcal{F}\left(  \frac{p-1}{q} \right).
    \]
    If $\frac{1}{2}i(i+1) \leq \frac{p}{q} < \frac{1}{2}(i+1)(i+2)$, where $i \geq 2$, then 
    \[
        \mathcal{F}\left(  \frac{p-1}{q} \right) \leq \mathcal{F}\left(  \frac{p}{q} \right) < \mathcal{F}\left(  \frac{1}{2}(i+1)(i+2) \right) = i+1,
    \]
    where Proposition \ref{Prop: Properties F} has been used. Hence, $\mathcal{F}\left(  \frac{p-1}{q} \right) \leq i$. The stronger inequality
    \[
        \frac{1}{2}i(i+1) \geq 1 + \frac{i}{2}, 
    \]
    holds for $i \geq 2$. For $\frac{p}{q} \in [1,3)$, we have 
    \[
       \mathcal{F}\left(  \frac{p-1}{q} \right) \leq 1, 
    \]
    and the stronger inequality yields
    \[
        \frac{p}{q} \geq \frac{3}{2}.
    \]
\end{proof}
\begin{Proposition}[Recurrence Relations for the $B(x)$ Family]
    The relation 
    \[
         T_k B(x) = x_k B(x),
    \]
    with $k \in \{1,2\}$ and $n \in \N_0$, can be written as 
    \[
        x_k B_{[n]}(\boldsymbol{x}) = b_{[n,n+1];k}B_{[n+1]}(\boldsymbol{x}) + b_{[n,n];k}B_{[n]}(\boldsymbol{x}) + \sum_{i = 1}^\alpha b_{[n,n-i];k}B_{[n-i]}(\boldsymbol{x}),
    \]
    where $\alpha = \left\lceil \frac{p}{q} \right\rceil$, $b_{[i,j];k} \in \R^{(i+1)q\times(j+1)q}$ and 
    \begin{align*}
        b_{[n+1,n];1} & = \begin{bNiceMatrix}[first-row,last-row,nullify-dots]
            & & \\
            \boxed{\cdot} & 0 & \Cdots & 0 & 0 & \Cdots & 0 \\
            \ast & \Ddots & \Ddots & \Vdots & \Vdots & & \Vdots \\
            \Vdots & \Ddots & & 0 \\
            \ast & \Cdots & \ast & \boxed{\cdot} & 0 & \Cdots & 0 \\
            & & 
            \CodeAfter
            \tikz \draw [solid] (5-|1) -- (5-|5);
            \tikz \node [below=1pt] at (5-|3) {$(n+1)q$};
            \tikz \draw [solid] (1-|5) -- (1-|8);
            \tikz \node [above=1pt] at (1-|6.5) {$q$};
            \tikz \draw [solid] (1-|5) -- (5-|5);
        \end{bNiceMatrix}, &  b_{[n+1,n];2} & = \begin{bNiceMatrix}[first-row,last-row,nullify-dots]
            & & \\
            \ast & \Cdots & \ast & \boxed{\cdot} & 0 & \Cdots & 0 \\
            \Vdots & & \Vdots & \ast & \Ddots & \Ddots & \Vdots \\
            \Vdots & & \Vdots & \Vdots & \Ddots & & 0 \\
            \ast & \Cdots & \ast & \ast & \Cdots & \ast & \boxed{\cdot} \\
             & & 
            \CodeAfter
            \tikz \draw [solid] (5-|1) -- (5-|4);
            \tikz \node [below=1pt] at (5-|2.5) {$q$};
            \tikz \draw [solid] (1-|4) -- (1-|8);
            \tikz \node [above=1pt] at (1-|6.5) {$(n+1)q$};
            \tikz \draw [solid] (1-|4) -- (5-|4);
        \end{bNiceMatrix}. 
    \end{align*}
    Once again, the matrices $b_{[i,j];k}$ are assumed to be zero whenever $j < 0$.
\end{Proposition}
\begin{proof}
We aim to show that $T_k$ can be expressed as:
\[
T_k = 
\begin{bNiceMatrix}
b_{[0,0];k} & b_{[0,1];k} & 0_{[0,2]} & \Cdots[shorten-end=5pt] \\
\Vdots & \Vdots & \Ddots & \Ddots[shorten-end=28pt] \\
b_{[\alpha,0];k} & b_{[\alpha,1];k} & \Cdots & b_{[\alpha,\alpha+1];k} & 0_{[\alpha,\alpha+2]} & \Cdots[shorten-end=5pt] \\
0_{[\alpha+1,0]} & \Ddots[shorten-end=5pt] & & & \Ddots[shorten-end=7pt] & \Ddots[shorten-end=4pt] \\
\Vdots[shorten-end=5pt] & \Ddots[shorten-end=5pt]
\end{bNiceMatrix}.
\]

Consider the submatrix
\[
\begin{bNiceMatrix}
0_{[0,n-1]} & \Cdots & 0_{[n-3,n-1]} & b_{[n-2,n-1];k} & \Cdots & b_{[n+\alpha-1,n-1];k}
\end{bNiceMatrix}^\top,
\]
whose rows in the original matrix range over 
$\{ 0, \dots, \frac{q}{2}(\alpha + n - 1)(\alpha + n + 2) + q - 1 \}$ 
and whose columns range over 
$\{ \frac{q}{2}n(n-1), \dots, \frac{q}{2}n(n+1) - 1 \}$.

For any given column, the position of the last nonzero entry of $T_k$ is determined by the pair $(m_{p;k}^+, m)$, where
\[
m_{p;k}^+ = m + pk + p \mathcal{F}\left( \frac{m}{p} \right).
\]
Given $n \in \N$, we must verify, for every column of the submatrix under consideration, that
\[
\frac{q}{2}(\alpha + n - 1)(\alpha + n + 2) + q - 1 
\geq \frac{q}{2}n(n-1) + i + kp + p \mathcal{F}\left( \frac{1}{2}\frac{q}{p}n(n-1) + \frac{i}{p} \right),
\]
where $i \in \{ 0, \dots, nq - 1 \}$.

By setting $\alpha = 1$ and interchanging $p \leftrightarrow q$, we recover the inequality analyzed in the proof of Proposition \ref{Prop: RR for A - Matrix}.  
We now focus on the case $k = 2$ and $i = nq - 1$, as this yields the most stringent inequality. In this situation, $\frac{q}{p} \leq 1$, and hence
\[
\mathcal{F}\left( \frac{1}{2}\frac{q}{p}n(n+1) - \frac{1}{p} \right)
< \mathcal{F}\left( \frac{1}{2}n(n+1) \right) = n.
\]
By the properties of $\mathcal{F}(x)$, we therefore have
\[
\mathcal{F}\left( \frac{1}{2}\frac{q}{p}n(n+1) - \frac{1}{p} \right) \leq n - 1.
\]
If the stronger inequality
\[
\frac{q}{2}(\alpha + n - 1)(\alpha + n + 2) + q - 1 
\geq \frac{q}{2}n(n+1) - 1 + 2p + p(n - 1),
\]
is satisfied, the claim follows.  
This inequality can be rearranged as
\[
\left\lceil \frac{p}{q} \right\rceil (2n + 1 + \left\lceil \frac{p}{q} \right\rceil)
\geq \frac{p}{q} (2n + 2),
\]
which always holds, since 
$\left\lceil \frac{p}{q} \right\rceil \geq \frac{p}{q} \geq 1$.
\end{proof}
\section{Christoffel--Darboux Kernels}

We now introduce the Christoffel--Darboux (CD) kernels in this setting. The definition mirrors that used in the context of Mixed Multiple Orthogonality, with the distinction that we now consider two variables.

\begin{Definition}
    The CD kernels are defined by
    \begin{equation*}
        K^{[n]} (\boldsymbol{x}, \boldsymbol{y}) \coloneq A^{[n]}(\boldsymbol{x})B^{[n]}(\boldsymbol{y}),
    \end{equation*}
    for $n \in \mathbb{N}_0$. The superscript $[n]$ denotes truncation up to the $(n+1)$-th row of $A(\boldsymbol{x})$ and the $(n+1)$-th column of $B(\boldsymbol{y})$. Thus, $K^{[n]}(\boldsymbol{x}, \boldsymbol{y})$ is a matrix polynomial of size $p \times q$, with entries given by
    \begin{equation*}
        K_{a,b}^{[n]}(\boldsymbol{x}, \boldsymbol{y}) = \sum_{i = 0}^n A_i^{(a)}(\boldsymbol{x})B^{(b)}_i(\boldsymbol{y}).
    \end{equation*}
\end{Definition}

\begin{Lemma} \label{Rm: AlternativeExpCDk}
    The truncation of $A(\boldsymbol{x})$ and $B(\boldsymbol{y})$, and hence the CD kernel itself, can be equivalently written as:
    \begin{equation*}
        K^{[n]} (\boldsymbol{x}, \boldsymbol{y}) \coloneq A(\boldsymbol{x}) \Pi_n B(\boldsymbol{y}),
    \end{equation*}
    where $\Pi_n$ is a semi-infinite matrix with ones on the first $n+1$ diagonal entries and zeros elsewhere.
\end{Lemma}

The CD kernels introduced above exhibit properties akin to those of projectors in Hilbert spaces. However, an exact analogue in the setting of a matrix of measures (or a suitable generalisation thereof) is yet to be fully developed.

\begin{Proposition}[Reproduction Property]
    The CD kernel satisfies the following reproduction property:
    \begin{equation*}
        \int_\Delta K^{[n]} (\boldsymbol{x}, \boldsymbol{t}) \, \mathrm{d} \mu (\boldsymbol{t})  K^{[n]} (\boldsymbol{t}, \boldsymbol{y}) = K^{[n]} (\boldsymbol{x}, \boldsymbol{y}).
    \end{equation*}
\end{Proposition}

\begin{proof}
    Using Lemma \ref{Rm: AlternativeExpCDk}, we compute:
\[    \begin{aligned}
        \int_\Delta A(\boldsymbol{x}) \Pi_n B(\boldsymbol{t}) \, \mathrm{d}\mu(\boldsymbol{t}) A(\boldsymbol{t}) \Pi_n B(\boldsymbol{y}) &= A(\boldsymbol{x}) \Pi_n \left( \int_\Delta B(\boldsymbol{t}) \, \mathrm{d}\mu(\boldsymbol{t}) A(\boldsymbol{t}) \right) \Pi_n B(\boldsymbol{y}) 
       \\& = A(\boldsymbol{x}) \Pi_n \Pi_n B(\boldsymbol{y}) \\&=  A(\boldsymbol{x}) \Pi_n B(\boldsymbol{y}) \\&= K^{[n]}(\boldsymbol{x}, \boldsymbol{y}).
    \end{aligned}\]
\end{proof}

\begin{Proposition}[Projection Property]
    Let $P(\boldsymbol{x})$ be a monic matrix polynomial of size $p \times p$ and grlex-degree $I \sim (i,j)$, i.e.,
    \begin{equation*}
        P(\boldsymbol{x}) = x_1^{i-j}x_2^j I_p + x_1^{i-j+1}x_2^{j-1}P_{I-1} + \dots + P_{0}.
    \end{equation*}
    Then, the following identity holds:
\[    \begin{aligned}
        \int_\Delta K^{[n]} (\boldsymbol{x}, \boldsymbol{y}) \, \mathrm{d} \mu (\boldsymbol{y})  P(\boldsymbol{y}) & = P(\boldsymbol{x}), &  n &\geq Ip + p - 1.
    \end{aligned}\]
\end{Proposition}

\begin{proof}
    Once again, we invoke Lemma \ref{Rm: AlternativeExpCDk}:
    \begin{multline*}
        \int_\Delta K^{[n]} (\boldsymbol{x}, \boldsymbol{y}) \, \mathrm{d} \mu (\boldsymbol{y})  P(\boldsymbol{y}) = A(\boldsymbol{x}) \Pi_n \int_{\Delta} B(\boldsymbol{y}) \, \mathrm{d} \mu(\boldsymbol{y}) X_{[p]}^\top (\boldsymbol{y}) \begin{bNiceMatrix}
            P_0 \\ P_1 \\ \Vdots \\ I_p \\ 0_p \\ \Vdots[shorten-end=3pt] 
        \end{bNiceMatrix} \\ 
        = A(\boldsymbol{x}) \Pi_n \bar{S}^{-\top} \begin{bNiceMatrix}
            P_0 \\ P_1 \\ \Vdots \\ I_p \\ 0_p \\ \Vdots[shorten-end=3pt] 
        \end{bNiceMatrix} 
        = A(\boldsymbol{x}) \Pi_n \begin{bNiceMatrix}
            \sum_{k=0}^{I-1} \left(\bar{S}^{-\top}\right)_{(0,k)} P_k + \left(\bar{S}^{-\top}\right)_{(0,I)} \\[3pt]
            \sum_{k=1}^{I-1} \left(\bar{S}^{-\top}\right)_{(1,k)} P_k + \left(\bar{S}^{-\top}\right)_{(1,I)} \\
            \Vdots \\
            \left(\bar{S}^{-\top}\right)_{(I,I)} \\
            0_P \\
            \Vdots[shorten-end=2pt] 
        \end{bNiceMatrix}.
    \end{multline*}
    For $n+1$ (i.e., the rank of $\Pi_n$) greater than $Ip + p$ (the number of nonzero rows), we have:
    \begin{equation*}
        \Pi_n \begin{bNiceMatrix}
            \sum_{k=0}^{I-1} \left(\bar{S}^{-\top}\right)_{(0,k)} P_k + \left(\bar{S}^{-\top}\right)_{(0,I)} \\[3pt]
            \sum_{k=1}^{I-1} \left(\bar{S}^{-\top}\right)_{(1,k)} P_k + \left(\bar{S}^{-\top}\right)_{(1,I)} \\
            \Vdots \\
            \left(\bar{S}^{-\top}\right)_{(I,I)} \\
            0_P \\
            \Vdots 
        \end{bNiceMatrix} = 
        \begin{bNiceMatrix}
            \sum_{k=0}^{I-1} \left(\bar{S}^{-\top}\right)_{(0,k)} P_k + \left(\bar{S}^{-\top}\right)_{(0,I)} \\[3pt]
            \sum_{k=1}^{I-1} \left(\bar{S}^{-\top}\right)_{(1,k)} P_k + \left(\bar{S}^{-\top}\right)_{(1,I)} \\
            \Vdots \\
            \left(\bar{S}^{-\top}\right)_{(I,I)} \\
            0_P \\
            \Vdots 
        \end{bNiceMatrix},
    \end{equation*}
    and therefore,
    \begin{multline*}
        \int_\Delta K^{[n]} (\boldsymbol{x}, \boldsymbol{y}) \, \mathrm{d} \mu (\boldsymbol{y})  P(\boldsymbol{y}) = 
        A(\boldsymbol{x}) \begin{bNiceMatrix}
            \sum_{k=0}^{I-1} \left(\bar{S}^{-\top}\right)_{(0,k)} P_k + \left(\bar{S}^{-\top}\right)_{(0,I)} \\[3pt]
            \sum_{k=1}^{I-1} \left(\bar{S}^{-\top}\right)_{(1,k)} P_k + \left(\bar{S}^{-\top}\right)_{(1,I)} \\
            \Vdots \\
            \left(\bar{S}^{-\top}\right)_{(I,I)} \\
            0_P \\
            \Vdots 
        \end{bNiceMatrix} 
        = A(\boldsymbol{x}) \bar{S}^{-\top} \begin{bNiceMatrix}
            P_0 \\ P_1 \\ \Vdots \\ I_p \\ 0_p \\ \Vdots[shorten-end=3pt] 
        \end{bNiceMatrix} \\
        = X_{[p]}^\top (\boldsymbol{x}) \begin{bNiceMatrix}
            P_0 \\ P_1 \\ \Vdots \\ I_p \\ 0_p \\ \Vdots[shorten-end=3pt] 
        \end{bNiceMatrix} = P(\boldsymbol{x}).
    \end{multline*}
\end{proof}

\begin{Corollary}
    A corresponding result holds for the right projection. Let $P(\boldsymbol{x})$ be a monic matrix polynomial of size $q \times q$ and grlex-degree $I \sim (i,j)$. Then,
 \[   \begin{aligned}
        \int_\Delta P(\boldsymbol{x}) \, \mathrm{d} \mu (\boldsymbol{x}) K^{[n]} (\boldsymbol{x}, \boldsymbol{y}) & = P(\boldsymbol{y}), & n & \geq Iq + q - 1.
    \end{aligned}\]
\end{Corollary}
For the recurrence matrices \(T_k\) and a given \(n \in \N_0\), we consider the following block decomposition:
\NiceMatrixOptions{cell-space-limits = 1.5pt}
\[
T_k = \left[\begin{NiceArray}{c|c}
	T_k^{[n]} & T_k^{[n,>n]} \\ \hline
	T_k^{[>n,n]} & T_k^{[>n]}
\end{NiceArray}\right],
\]
where \(T_k^{[n]}\), \(T_k^{[n,>n]}\), and \(T_k^{[>n,n]}\) are matrices of dimensions \((n+1) \times (n+1)\), \((n+1) \times \infty\), and \(\infty \times (n+1)\), respectively.

Before presenting general results concerning the structure of \(T_k^{[n,>n]}\) and \(T_k^{[>n,n]}\), let us first examine a representative example.

\begin{Example}
As in the earlier case study, take \(p = 2\) and \(q = 1\). The decomposition of the matrix \(T_1\) is then given by:
\[
	T_1 = \left[\begin{NiceMatrix}
	\ast & 1 & 0 & 0 & 0 & \Cdots[shorten-end=5pt]  \\
	\ast & \ast & \ast & 1 & 0 & \Cdots[shorten-end=5pt]  \\
	\boxed{\cdot} & \ast & \ast & \ast & 1 & 0 & \Cdots[shorten-end=5pt]  \\
	0 & \boxed{\cdot} & \ast & \ast & \ast & \ast & 1 & 0 & \Cdots[shorten-end=5pt]  \\
	0 & 0 & \ast & \ast & \ast & \ast & \ast & 1 & 0 & \Cdots[shorten-end=5pt]  \\
	0 & 0 & \ast & \ast & \ast & \ast & \ast & \ast & 1 & 0 & \Cdots[shorten-end=5pt]  \\
	0 & 0 & \boxed{\cdot} & \ast & \ast & \Cdots[shorten-end=5pt]  \\
	0 & 0 & 0 & \boxed{\cdot} & \ast & \Cdots[shorten-end=5pt]\\
	\Vdots[shorten-end=3pt] & \Vdots[shorten-end=3pt] & \Vdots[shorten-end=3pt] & \Vdots[shorten-end=3pt]
	\CodeAfter
	\tikz \draw [solid] (5-|1) -- (5-|last); % Horizontal line after row 3 (for n=2)
    \tikz \draw [solid] (1-|5) -- (last-|5); % Vertical line after column 4 (for n=2)
\end{NiceMatrix}\right],
\]
where \(n = 3\). The matrices \(T_1^{[>3,3]}\) and \(T_1^{[3,>3]}\) are given by:
\begin{align*}
	T_1^{[>3,3]} & = \begin{bNiceMatrix}
		0 & 0 & \ast & \ast \\
		0 & 0 & \ast & \ast \\
		0 & 0 & \boxed{\cdot} & \ast \\
		0 & 0 & 0 & \boxed{\cdot} \\
		0 & 0 & 0 & 0 \\
		\Vdots[shorten-end=5pt] & \Vdots[shorten-end=5pt] & \Vdots[shorten-end=5pt] & \Vdots[shorten-end=5pt]
	\end{bNiceMatrix}, & T_1^{[3,>3]} & = \begin{bNiceMatrix}
		0 & \Cdots[shorten-end=7pt] \\
		1 & 0 & \Cdots[shorten-end=7pt] \\
		\ast & \ast & 1 & 0 & \Cdots[shorten-end=7pt]
	\end{bNiceMatrix}.
\end{align*}
Both matrices can equivalently be expressed as:
\begin{align*}
	T_1^{[>3,3]} & = \left[\begin{NiceArray}{c|c}
		0 & 	\mathscr{T}_1^{[>3,3]}	\\ \hline
		0 & 0
	\end{NiceArray}\right], & \mathscr{T}_1^{[>3,3]} & = \begin{bNiceMatrix}
		\ast & \ast \\
		\ast & \ast \\
		\boxed{\cdot} & \ast \\
		0 & \boxed{\cdot}
	\end{bNiceMatrix}, \\
	T_1^{[3,>3]} & = \left[\begin{NiceArray}{c|c}
		0 & 	0	\\ \hline
		\mathscr{T}_1^{[>3,3]} & 0
	\end{NiceArray}\right], & \mathscr{T}_1^{[3,>3]} & = \begin{bNiceMatrix}
	1 & 0 & 0 \\
	\ast & \ast & 1
	\end{bNiceMatrix}.
\end{align*}
In both cases, the dimensions of the zero blocks are determined by the size of \(\mathscr{T}_1^{[>3,3]}\) and \(\mathscr{T}_1^{[3,>3]}\).
\end{Example}

In the general case, a similar decomposition holds for \(T_k^{[>n,n]}\) and \(T_k^{[n,>n]}\). For any arbitrary \(n \in \N_0\), we have the following characterization: for \(\mathscr{T}_k^{[>n,n]}\), the \((0,0)\) entry coincides with \(T_{n+1,(N+1)_{p;k}^-}\), as given by the recurrence relation:
\[
	x_k B^{(b)}_{n+1}(\boldsymbol{x}) = B^{(b)}_{(n+1)^+_{q;k}}(\boldsymbol{x}) + \sum_{i=(N+1)_{p;k}^-}^{(n+1)_{q;k}^+-1} T_{k;\,n+1,i}B^{(b)}_{i}(\boldsymbol{x}),
\]
introduced in Proposition \ref{Recurrence Relation-GeneralCase}. Likewise, the bottom-right entry \(T_{k;\,n^+_{p;k},n}\) (equivalently, \(\frac{H_{n^+_{p;k}}}{H_n}\)) is given by:
\[
	x_k A^{(a)}_n(\boldsymbol{x}) = T_{k;\,n^+_{p;k},n}A^{(a)}_{n^+_{p;k}}(\boldsymbol{x}) + \sum_{i=N^-_{q;k}}^{n^+_{p;k}-1} T_{k;\,i,n}A^{(a)}_{i}(\boldsymbol{x}).
\]
The dimensions of \(\mathscr{T}_k^{[>n,n]}\) are \((n^+_{p;k}-n+1) \times (n-(N+1)_{p;k}^-+1)\). Similar reasoning for \(\mathscr{T}_k^{[n,>n]}\) shows that the \((0,0)\) entry is \(T_{(n+1)_p^-,n+1}\) (always equal to \(1\)), and the bottom-right entry is \(T_{n,n^+_{q;k}}\). The size of this matrix is \(((N+1)^-_{q;k}-n+1) \times (n^+_{q;k}-n)\).

We also introduce the notation:
\[
\begin{aligned}
	A &= \begin{bNiceArray}{c|c}
		A^{[n]} & A^{[>n]}
	\end{bNiceArray}, &
	B &= \left[\begin{NiceArray}{c}
		B^{[n]} \\ \hline B^{[>n]}
	\end{NiceArray}\right],
\end{aligned}
\]
where
\[\begin{aligned}
    A^{[>n]} & = \begin{bNiceArray}{c|c}
        \mathscr{A}^{[>n;k]}(\boldsymbol{x}) & A^{[>>n;k]}(\boldsymbol{x})
    \end{bNiceArray}, & \mathscr{A}^{[>n;k]}(\boldsymbol{x}) & = \begin{bNiceMatrix}
    A_{n+1}^{(1)}(\boldsymbol{x}) & \Cdots & A_{n^+_{p;k}}^{(1)}(\boldsymbol{x}) \\
    \Vdots & & \Vdots \\
    A_{n+1}^{(p)}(\boldsymbol{x}) & \Cdots & A_{n^+_{p;k}}^{(p)}(\boldsymbol{x})
    \end{bNiceMatrix}, \\
     A^{[n]} & = \begin{bNiceArray}{c|c}
         A^{[<n;k]}(\boldsymbol{x}) & \mathscr{A}^{[n;k]}(\boldsymbol{x})
    \end{bNiceArray}, & \mathscr{A}^{[n;k]}(\boldsymbol{x}) & = \begin{bNiceMatrix}
    A_{(N+1)_{p;k}^-}^{(1)}(\boldsymbol{x}) & \Cdots & A_{n}^{(1)}(\boldsymbol{x}) \\
    \Vdots & & \Vdots \\
    A_{(N+1)_{p;k}^-}^{(p)}(\boldsymbol{x}) & \Cdots & A_{n}^{(p)}(\boldsymbol{x})
    \end{bNiceMatrix}.
\end{aligned}\]
For \(B\), we define:
\[\begin{aligned}
B^{[>n]} & = \begin{bNiceArray}{c} 
        \mathscr{B}^{[>n;k]}(\boldsymbol{x}) \\ \hline
        B^{[>>n;k]}(\boldsymbol{x})
    \end{bNiceArray}, & \mathscr{B}^{[>n;k]}(\boldsymbol{x}) & = \begin{bNiceMatrix}
        B_{n+1}^{(1)}(\boldsymbol{x}) & \Cdots & B_{n+1}^{(q)}(\boldsymbol{x}) \\
        \Vdots & & \Vdots \\
        B_{n_{q;k}^+}^{(1)}(\boldsymbol{x}) & \Cdots & B_{n_{q;k}^+}^{(q)}(\boldsymbol{x})
    \end{bNiceMatrix}, \\
    B^{[n]} & = \begin{bNiceArray}{c} 
        B^{[<n;k]}(\boldsymbol{x}) \\ \hline
        \mathscr{B}^{[n;k]}(\boldsymbol{x})
    \end{bNiceArray}, & \mathscr{B}^{[n;k]}(\boldsymbol{x}) & = \begin{bNiceMatrix}
        B_{(N+1)^-_{q;k}}^{(1)}(\boldsymbol{x}) & \Cdots & B_{(N+1)^-_{q;k}}^{(q)}(\boldsymbol{x}) \\
        \Vdots & & \Vdots \\
        B_{n}^{(1)}(\boldsymbol{x}) & \Cdots & B_{n}^{(q)}(\boldsymbol{x})
    \end{bNiceMatrix}.
\end{aligned}\]

\begin{Theorem}[Christoffel--Darboux formula]\label{KcomoT}
For $k\in\{1,2\}$, the following matrix Christoffel–Darboux identities hold:
\[
(x_k - y_k)K^{[n]}(\boldsymbol{x}, \boldsymbol{y}) = \mathscr{A}^{[>n;k]}(\boldsymbol{x})\mathscr{T}_k^{[>n,n]}\mathscr{B}^{[n;k]}(\boldsymbol{y}) - \mathscr{A}^{[n;k]}(\boldsymbol{x})\mathscr{T}_k^{[n,>n]}\mathscr{B}^{[>n;k]}(\boldsymbol{y}).
\]
\end{Theorem}

\begin{proof}
From Proposition \ref{Recurrence Relation-GeneralCase}, we have \((A(\boldsymbol{x})T_k)^{[n]} = x_k A^{[n]}(\boldsymbol{x})\) and \((T_k B(\boldsymbol{y}))^{[n]} = y_k B^{[n]}(\boldsymbol{y})\). Expanding:
\[
\begin{aligned}
	x_k A^{[n]}(\boldsymbol{x}) &= A^{[n]}(\boldsymbol{x}) T_k^{[n]} + A^{[>n]}(\boldsymbol{x}) T_k^{[>n,n]}, \\
	y_k B^{[n]}(\boldsymbol{y}) &= T_k^{[n]} B^{[n]}(\boldsymbol{y}) + T_k^{[n,>n]} B^{[>n]}(\boldsymbol{y}).
\end{aligned}
\]
Thus,
\[
\begin{aligned}
	(x_k - y_k) K^{[n]}(\boldsymbol{x}, \boldsymbol{y}) &= A^{[>n]}(\boldsymbol{x}) T_k^{[>n,n]} B^{[n]}(\boldsymbol{y}) - A^{[n]}(\boldsymbol{x}) T_k^{[n,>n]} B^{[>n]}(\boldsymbol{y}) \\
	&= \mathscr{A}^{[>n;k]}(\boldsymbol{x}) \mathscr{T}_k^{[>n,n]} \mathscr{B}^{[n;k]}(\boldsymbol{y}) - \mathscr{A}^{[n;k]}(\boldsymbol{x}) \mathscr{T}_k^{[n,>n]} \mathscr{B}^{[>n;k]}(\boldsymbol{y}).
\end{aligned}
\]
\end{proof}
\begin{Example}
    Let us continue with the example, explicitly: 
 \[   \begin{aligned}
        \mathscr{T}_1^{[>3,3]} & = \begin{bNiceMatrix}
            T_{1;4,2} & T_{1;4,3} \\
            T_{1;5,2} & T_{1;5,3} \\
            T_{1;6,2} & T_{1;6,3} \\
            0 & T_{1;7,3} 
        \end{bNiceMatrix} & \mathscr{T}_1^{[3,>3]} & = \begin{bNiceMatrix}
            1 & 0 & 0 \\
            T_{1;3,4} & T_{1;3,5} & 1
        \end{bNiceMatrix}.
    \end{aligned}\]
    Lastly,
    \begin{multline*}
        (x_1 - y_1) K_a^{[3]}(\boldsymbol{x}, \boldsymbol{y}) = \begin{bNiceMatrix}
            A_4^{(a)}(\boldsymbol{x}) & A_5^{(a)}(\boldsymbol{x}) &A_6^{(a)}(\boldsymbol{x}) & A_7^{(a)}(\boldsymbol{x}) 
        \end{bNiceMatrix} \begin{bNiceMatrix}
            T_{1;4,2} & T_{1;4,3} \\
            T_{1;5,2} & T_{1;5,3} \\
            T_{1;6,2} & T_{1;6,3} \\
            0 & T_{1;7,3} 
        \end{bNiceMatrix} \begin{bNiceMatrix}
            B_2(\boldsymbol{y}) \\
            B_3(\boldsymbol{y})
        \end{bNiceMatrix} \\ - \begin{bNiceMatrix}
            A_2^{(a)}(\boldsymbol{x}) & A_3^{(a)}(\boldsymbol{x})
        \end{bNiceMatrix} \begin{bNiceMatrix}
            1 & 0 & 0 \\
            T_{1;3,4} & T_{1;3,5} & 1
        \end{bNiceMatrix} \begin{bNiceMatrix}
            B_4(\boldsymbol{y}) \\
            B_5(\boldsymbol{y}) \\
            B_6(\boldsymbol{y})
        \end{bNiceMatrix}.
    \end{multline*}
\end{Example}
Finally, we establish an analogous Aitken--Berg--Collar (ABC) theorem for bivariate multiple orthogonal polynomials of mixed type.

\begin{Lemma} \label{Rm: Truncation LU}
If a semi-infinite matrix admits an LU factorization, then its \(n\)-th truncation satisfies:
\[
	\mathscr{M}^{[n]} = \mathscr{L}^{[n]} \mathscr{U}^{[n]}.
\]
\end{Lemma}

\begin{Theorem}[Aitken--Berg--Collar]
The following identity holds:
\[
	K^{[n]}(\boldsymbol{x}, \boldsymbol{y}) = \left(X_{[p]}(\boldsymbol{x})^\top\right)^{[n]} \left( \mathscr{M}^{[n]} \right)^{-1} \left(X_{[q]}(\boldsymbol{y})\right)^{[n]},
\]
where the superscript \({}^{[n]}\) denotes the \(n\)-th truncation of a semi-infinite matrix, or the appropriate truncation to a \(r \times n\) or \(n \times r\) matrix when applied to a \(r \times \infty\) or \(\infty \times r\) matrix, respectively.
\end{Theorem}

\begin{proof}
Using Lemma \ref{Rm: AlternativeExpCDk}, we have:
\[
	K^{[n]}(\boldsymbol{x}, \boldsymbol{y}) = A(\boldsymbol{x}) \Pi_n B(\boldsymbol{y}) = A(\boldsymbol{x}) \Pi_n \Pi_n B(\boldsymbol{y}),
\]
since \(\Pi_n\) is nilpotent. Decomposing \(A(x)\) and \(B(x)\), we obtain:
\[
	K^{[n]}(\boldsymbol{x}, \boldsymbol{y}) = X_{[p]}^\top(\boldsymbol{x}) \bar{S}^\top \Pi_n \Pi_n H^{-1} S X_{[q]}(\boldsymbol{y}).
\]
It is easy to verify:
\begin{align*}
	\bar{S}^\top \Pi_n & = \left[\begin{NiceArray}{c|c}
		(\bar{S}^\top)^{[n]} & 0 \\ \hline
		0 & 0
	\end{NiceArray}\right], &
	\Pi_n H^{-1} S & = \left[\begin{NiceArray}{c|c}
		(H^{-1} S)^{[n]} & 0 \\ \hline
		0 & 0
	\end{NiceArray}\right],
\end{align*}
so the CD kernels can be rewritten as:
\[
	K^{[n]}(\boldsymbol{x}, \boldsymbol{y}) = X_{[p]}^\top(\boldsymbol{x}) \left[\begin{NiceArray}{c|c}
		(\mathscr{M}^{[n]})^{-1} & 0 \\ \hline
		0 & 0
	\end{NiceArray}\right] X_{[q]}(\boldsymbol{y}),
\]
where Lemma \ref{Rm: Truncation LU} has been used. The result follows.
\end{proof}
\section{A Study Case: Bivariate Jacobi--Piñeiro of the Mixed-Type on the Triangle}
We now specialize to a concrete family of weights, namely the Jacobi–Piñeiro class on the triangular domain. Our aims are the following:
\begin{itemize}
    \item To specify the measure matrix and the associated monomials.
    \item To construct a truncation of the moment matrix and to perform the Gauss–Borel factorization.
    \item To construct both families of orthogonal polynomials up to the prescribed truncation order.
\end{itemize}

Fix parameters \(\alpha,\beta_a^p,\beta_b^q,\gamma_a^p,\gamma_b^q>-1\), for \(a \in \{1, \dots, p\}\) and \(b \in \{1, \dots, q\}\), and consider the triangle support
\[
\Delta := \{(x,y)\in\mathbb R^2 : x>0,\ y>0,\ x+y<1\}.
\]
We define two bivariate Jacobi–Piñeiro weight functions, \(w_{(q)}(\boldsymbol{x})\) and \(v_{(p)}(\boldsymbol{x})\),
\[\begin{aligned}
    w_{(q)}(\boldsymbol x) & \coloneq \begin{bNiceMatrix} y^{\gamma_1^q}(1-x-y)^{\beta_1^q}\\ \Vdots \\ y^{\gamma_q^q}(1-x-y)^{\beta_q^q} \end{bNiceMatrix}, 
    & 
    v_{(p)}(\boldsymbol x) & \coloneq \begin{bNiceMatrix} y^{\gamma_1^p}(1-x-y)^{\beta_1^p} & \Cdots & y^{\gamma_p^p}(1-x-y)^{\beta_p^p} \end{bNiceMatrix}.
\end{aligned}\]
The \(q\times p\) measure is then given by
\[
\d\mu(\boldsymbol{x})
=
x^{\alpha} w_{(q)}(\boldsymbol x) v_{(p)}(\boldsymbol x)\d x\d y .
\]

The moments are defined by
\[\begin{aligned}
    m^{i,j}_{b,a} & \coloneq \int_\Delta x^i y^j \d\mu_{b,a}(\boldsymbol x),
    & 
    m^{i,j} & \coloneq \int_\Delta x^i y^j \d\mu(\boldsymbol x),
\end{aligned}\]
and may be computed explicitly as
\[
    m^{i,j}_{b,a} = 
    \frac{\Gamma\left( \alpha + i + 1 \right)
    \Gamma\left( j + \gamma_b^q + \gamma_a^p + 1\right)
    \Gamma\left( \beta_b^q + \beta_a^p + 1\right)}
    {\Gamma\left( \alpha + i + j + \gamma_b^q + \gamma_a^p + \beta_b^q + \beta_a^p + 3 \right)},
\]
where \(\Gamma(x)\) denotes the Euler Gamma function.

Consider a truncation of the moment matrix of size \(npq \times npq\), with \(n \in \mathbb N_0\):
\[
\mathscr{M}^{[npq]} = \int_\Delta \left(X_{[q]}(\boldsymbol{x})\right)^{[npq]} \, d\mu(\boldsymbol{x}) \, \left(X_{[p]}(\boldsymbol{x})^\top\right)^{[npq]}.
\]
For a given choice of parameters, whenever the Gauss–Borel factorization exists at this truncation, we obtain
\[
\mathscr{M}^{[npq]} = \left( S^{-1} \right)^{[npq]} \, H^{[npq]} \, \left( \bar{S}^{-\top} \right)^{[npq]}.
\]

We have developed Maple code that constructs this moment matrix, performs the truncated Gauss–Borel factorization, and computes the corresponding families of orthogonal polynomials. The code, available to any user with an active Maple licence, generates the polynomial families once the parameters \(p\), \(q\), \(n\), and the various exponents are specified, provided that the Gauss–Borel factorization exists at the chosen truncation. We have examined three representative examples:
\begin{itemize}
    \item \underline{Multiple case}: \(q=1, p=2, n=5\), with exponents \(\alpha=0\), \(\beta_1^p=1\), \(\beta_2^p=0\), \(\gamma_1^p=\frac{1}{2}\), \(\gamma_2^p=1\). The measure is
    \[
\d\mu(\boldsymbol{x})
=
\begin{bmatrix}
	y\sqrt{1-x-y} & 1-x-y
\end{bmatrix}\d x\d y.
\]
    \item \underline{Matrix case}: \(q=2, p=2, n=4\), with exponents \(\alpha = 0\), \( \{ \gamma_b^q \}_{b=1}^2 = \{ 0, \frac{1}{2} \}\), \( \{ \beta_b^q \}_{b=1}^2 = \{ 0, 1 \}\), \( \{ \gamma_a^p \}_{a=1}^2 = \{ 1, \frac{1}{2} \}\), \( \{ \beta_a^p \}_{a=1}^2 = \{ 1, \frac{1}{2} \}\). The measure is
    \[
\d\mu(\boldsymbol{x})
	=
	\begin{bmatrix}
		y(1-x-y) & \sqrt{y}\sqrt{1-x-y}\\
		y^{3/2}(1-x-y)^2 & y(1-x-y)^{\frac{3}{2}}
	\end{bmatrix}\d x\d y.
	\]
    \item \underline{Multiple mixed-type case}: \(q=2, p=3, n=2\), with exponents \(\alpha = 0\), \(\{ \gamma_b^q \}_{b=1}^2 = \{\frac{1}{2},0\}\), \(\{ \beta_b^q \}_{b=1}^2 = \{1,\frac{1}{2}\}\), \(\{ \gamma_a^p \}_{a=1}^3 = \{0, \frac{1}{4}, \frac{1}{2}\}\), \(\{ \beta_a^p \}_{a=1}^3 = \{0, \frac{3}{2}, \frac{1}{4}\}\). The measure is
    \[
	\d\mu(\boldsymbol{x})
	=
	\begin{bmatrix}
		\sqrt{y}(1-x-y) & y^{\frac{3}{4}}(1-x-y)^{\frac{5}{2}} & y(1-x-y)^{\frac{5}{4}} \\
		\sqrt{1-x-y} & y^{\frac{1}{4}}(1-x-y)^{2} & \sqrt{y}(1-x-y)^{\frac{3}{4}}
	\end{bmatrix}\d x\d y.
	\]
\end{itemize}

For all three examples, the Gauss–Borel factorization exists up to the \(npq\) truncation. Numerical tests performed in Python confirm that the leading principal minors remain non-vanishing for larger values of \(n\), although numerical underflow occurs when these determinants reach magnitudes of order \(10^{-324}\), at which point Python treats them as zero. Consequently, only analytic methods can establish orthogonality for all \(n \in \mathbb N_0\). Furthermore, the parameters were chosen so that \(\gamma_i^r - \gamma_j^r \notin \mathbb Z\) and \(\beta_i^r - \beta_j^r \notin \mathbb Z\) for \(r \in \{p,q\}\), as these conditions are required for the univariate Jacobi–Piñeiro polynomials to exist.

Given the extension of the formulas involved we only reproduce here two examples.
\begin{enumerate}
	\item 
For the multiple non-mixed type case , we take \(q=1,\, p=2\), the truncated moment matrix \(\mathscr M^{[10]}\in \mathbb R^{10\times 10}\) is
\begin{align*}
\mathscr M^{[10]} =
\begin{bmatrix}
	\frac{8}{105} & \frac{1}{6} & \frac{16}{945} & \frac{1}{24} & \frac{32}{945} & \frac{1}{24} & \frac{64}{10395} & \frac{1}{60} & \frac{64}{10395} & \frac{1}{120}\\[2pt]
	\frac{16}{945} & \frac{1}{24} & \frac{64}{10395} & \frac{1}{60} & \frac{64}{10395} & \frac{1}{120} & \frac{128}{45045} & \frac{1}{120} & \frac{256}{135135} & \frac{1}{360}\\[2pt]
	\frac{32}{945} & \frac{1}{24} & \frac{64}{10395} & \frac{1}{120} & \frac{64}{3465} & \frac{1}{60} & \frac{256}{135135} & \frac{1}{360} & \frac{128}{45045} & \frac{1}{360}\\[2pt]
	\frac{64}{10395} & \frac{1}{60} & \frac{128}{45045} & \frac{1}{120} & \frac{256}{135135} & \frac{1}{360} & \frac{1024}{675675} & \frac{1}{210} & \frac{512}{675675} & \frac{1}{840}\\[2pt]
	\frac{64}{10395} & \frac{1}{120} & \frac{256}{135135} & \frac{1}{360} & \frac{128}{45045} & \frac{1}{360} & \frac{512}{675675} & \frac{1}{840} & \frac{512}{675675} & \frac{1}{1260}\\[2pt]
	\frac{64}{3465} & \frac{1}{60} & \frac{128}{45045} & \frac{1}{360} & \frac{512}{45045} & \frac{1}{120} & \frac{512}{675675} & \frac{1}{1260} & \frac{1024}{675675} & \frac{1}{840}\\[2pt]
	\frac{128}{45045} & \frac{1}{120} & \frac{1024}{675675} & \frac{1}{210} & \frac{512}{675675} & \frac{1}{840} & \frac{2048}{2297295} & \frac{1}{336} & \frac{4096}{11486475} & \frac{1}{1680}\\[2pt]
	\frac{256}{135135} & \frac{1}{360} & \frac{512}{675675} & \frac{1}{840} & \frac{512}{675675} & \frac{1}{1260} & \frac{4096}{11486475} & \frac{1}{1680} & \frac{1024}{3828825} & \frac{1}{3360}\\[2pt]
	\frac{128}{45045} & \frac{1}{360} & \frac{512}{675675} & \frac{1}{1260} & \frac{1024}{675675} & \frac{1}{840} & \frac{1024}{3828825} & \frac{1}{3360} & \frac{4096}{11486475} & \frac{1}{3360}\\[2pt]
	\frac{512}{45045} & \frac{1}{120} & \frac{1024}{675675} & \frac{1}{840} & \frac{1024}{135135} & \frac{1}{210} & \frac{4096}{11486475} & \frac{1}{3360} & \frac{2048}{2297295} & \frac{1}{1680}
\end{bmatrix}.
\end{align*}
Its $LU$ factorization is: 
\begin{multline*}
	\left( S^{-1} \right)^{[10]} = \begin{bmatrix}
		1 & 0 & 0 & 0 & 0 & 0 & 0 & 0 & 0 & 0\\[2pt]
		\frac{2}{9} & 1 & 0 & 0 & 0 & 0 & 0 & 0 & 0 & 0\\[2pt]
		\frac{4}{9} & -7 & 1 & 0 & 0 & 0 & 0 & 0 & 0 & 0\\[2pt]
		\frac{8}{99} & \frac{38}{55} & -\frac{42}{3575} & 1 & 0 & 0 & 0 & 0 & 0 & 0\\[2pt]
		\frac{8}{99} & -\frac{61}{55} & \frac{739}{3575} & -\frac{254}{87} & 1 & 0 & 0 & 0 & 0 & 0\\[2pt]
		\frac{8}{33} & -\frac{282}{55} & \frac{2558}{3575} & \frac{217}{87} & -\frac{48}{5} & 1 & 0 & 0 & 0 & 0\\[2pt]
		\frac{16}{429} & \frac{327}{715} & -\frac{49}{3575} & \frac{235}{203} & -\frac{1}{175} & -\frac{1}{1820} & 1 & 0 & 0 & 0\\[2pt]
		\frac{32}{1287} & -\frac{211}{715} & \frac{727}{10725} & -\frac{872}{609} & \frac{281}{525} & \frac{11}{5460} & -\frac{1367}{1011} & 1 & 0 & 0\\[2pt]
		\frac{16}{429} & -\frac{531}{715} & \frac{1327}{10725} & -\frac{249}{203} & -\frac{389}{525} & \frac{751}{5460} & \frac{2753}{1011} & \frac{375}{212} & 1 & 0\\[2pt]
		\frac{64}{429} & -\frac{2553}{715} & \frac{1751}{3575} & \frac{610}{203} & -\frac{2011}{175} & \frac{317}{260} & \frac{1777}{337} & \frac{11367}{1060} & \frac{451}{435} & 1
	\end{bmatrix}, \\	\left( H\bar{S}^{-\top} \right)^{[10]} = \begin{bmatrix}
		\frac{8}{105} & \frac{1}{6} & \frac{16}{945} & \frac{1}{24} & \frac{32}{945} & \frac{1}{24} & \frac{64}{10395} & \frac{1}{60} & \frac{64}{10395} & \frac{1}{120}\\[2pt]
		0 & \frac{1}{216} & \frac{32}{13365} & \frac{1}{135} & -\frac{128}{93555} & -\frac{1}{1080} & \frac{256}{173745} & \frac{1}{216} & \frac{128}{243243} & \frac{1}{1080}\\[2pt]
		0 & 0 & \frac{32}{2079} & \frac{1}{24} & -\frac{64}{10395} & -\frac{1}{120} & \frac{256}{27027} & \frac{1}{36} & \frac{512}{135135} & \frac{1}{180}\\[2pt]
		0 & 0 & 0 & \frac{29}{85800} & \frac{128}{4129125} & -\frac{61}{1287000} & \frac{512}{4601025} & \frac{163}{300300} & -\frac{9472}{161035875} & -\frac{47}{819000}\\[2pt]
		0 & 0 & 0 & 0 & -\frac{64}{1306305} & -\frac{1}{31320} & \frac{1024}{3918915} & \frac{1}{1218} & -\frac{2176}{19594575} & -\frac{37}{219240}\\[2pt]
		0 & 0 & 0 & 0 & 0 & -\frac{1}{1350} & \frac{512}{225225} & \frac{1}{140} & -\frac{1024}{1126125} & -\frac{29}{18900}\\[2pt]
		0 & 0 & 0 & 0 & 0 & 0 & -\frac{43136}{5226346125} & -\frac{1}{573300} & \frac{27136}{5226346125} & \frac{1}{573300}\\[2pt]
		0 & 0 & 0 & 0 & 0 & 0 & 0 & -\frac{53}{3566808} & -\frac{420352}{135482972625} & \frac{1}{1455840}\\[2pt]
		0 & 0 & 0 & 0 & 0 & 0 & 0 & 0 & \frac{3712}{338212875} & -\frac{1}{305280}\\[2pt]
		0 & 0 & 0 & 0 & 0 & 0 & 0 & 0 & 0 & \frac{61}{2192400}
	\end{bmatrix}.
\end{multline*}
The type–II polynomials are defined by
\begin{multline*}
	\left( B(\boldsymbol{x}) \right)^{[10]}  \coloneq 
	\left( S \right)^{[10]} \left( X_{[1]}(\bm x) \right)^{[10]}
= \\ \begin{bmatrix}
		1\\[2pt]
		-\frac{2}{9} + x\\[2pt]
		-2 + 7x + y\\[2pt]
		\frac{16}{325} - \frac{2176}{3575}x + \frac{42}{3575}y + x^{2}\\[2pt]
		\frac{20}{87} - \frac{184}{87}x - \frac{5}{29}y + \frac{254}{87}x^{2} + x y\\[2pt]
		\frac{32}{15} - \frac{56}{3}x - \frac{12}{5}y + \frac{383}{15}x^{2} + \frac{48}{5}x y + y^{2}\\[2pt]
		-\frac{8}{455} + \frac{146}{455}x - \frac{1}{455}y - \frac{293}{260}x^{2} + \frac{1}{91}x y + \frac{1}{1820}y^{2} + x^{3}\\[2pt]
		-\frac{1256}{35385} + \frac{19552}{35385}x + \frac{1528}{35385}y - \frac{20122}{11795}x^{2} - \frac{6366}{11795}x y - \frac{3}{2359}y^{2} + \frac{1367}{1011}x^{3} + x^{2}y\\[2pt]
		\frac{74}{795} - \frac{456}{265}x + \frac{6}{265}y + \frac{3161}{530}x^{2} + \frac{183}{530}x y - \frac{29}{212}y^{2} - \frac{3253}{636}x^{3} - \frac{375}{212}x^{2}y + x y^{2}\\[2pt]
		\frac{80}{261} - \frac{784}{145}x - \frac{8}{145}y + \frac{262}{15}x^{2} + \frac{748}{145}x y - \frac{16}{15}y^{2} - \frac{18883}{1305}x^{3} - \frac{1289}{145}x^{2}y - \frac{451}{435}x y^{2} + y^{3}
	\end{bmatrix}.
\end{multline*}
The type-I polynomials are defined as follows:
\[
    \left( A(\boldsymbol{x}) \right)^{[10]}
	\coloneq \left( X_{[2]}(\boldsymbol{x})^\top \right)^{[10]} \left( \bar{S}^\top H^{-1} \right)^{[10]} = \begin{bNiceMatrix}
		A_{1}^{(1)} & \Cdots & A_{10}^{(1)}\\[2pt]
		A_{1}^{(2)} & \Cdots & A_{10}^{(2)}
	\end{bNiceMatrix},
\]
whit explicit expressions given by,
\begin{align*}
	A_{1}^{(1)} &= \frac{105}{8},\ A_{2}^{(1)} = -\frac{945}{2},\ A_{3}^{(1)} = \frac{945}{16} + \frac{2079}{32}x,\ A_{4}^{(1)} = \frac{675675}{464} - \frac{7432425}{928}x,\\ A_{5}^{(1)} &= \frac{63063}{4} - \frac{423423}{32}x - \frac{1306305}{64}y,\ A_{6}^{(1)} = -\frac{3465}{32} + \frac{45045}{128}x + \frac{225225}{256}y,\\
	A_{7}^{(1)} &= \frac{2599772175}{5392} - \frac{3832653825}{10784}x
	- \frac{8710576875}{21568}y - \frac{5226346125}{43136}x^{2},\\
	A_{8}^{(1)} &= \frac{672945273}{848} - \frac{2344997655}{3392}x
	- \frac{4460130675}{6784}y + \frac{48243195}{3392}x^{2},\\
	A_{9}^{(1)} &= \frac{69366297}{464} - \frac{183167985}{928}x
	- \frac{238363125}{1856}y + \frac{56921865}{928}x^{2}
	+ \frac{338212875}{3712}x y,\\
	A_{10}^{(1)} &= \frac{471936465}{7808} - \frac{1159233075}{15616}x
	- \frac{1525899375}{31232}y + \frac{225900675}{15616}x^{2}
	+ \frac{670044375}{62464}x y,\\[4pt]
	A_{1}^{(2)} &= 0,\ A_{2}^{(2)}= 216,\	A_{3}^{(2)} = -\frac{168}{5},\ A_{4}^{(2)}= -\frac{17160}{29} + \frac{85800}{29}x,\ A_{5}^{(2)}= -2184 + 1872x,\\
	A_{6}^{(2)} &= 240 - 270x - 1350y,\ A_{7}^{(2)} = -\frac{3712800}{337} + \frac{8353800}{337}x
	- \frac{125307000}{337}y,\\
	A_{8}^{(2)} &= -\frac{1120392}{53} + \frac{4176900}{53}x
	- \frac{32080860}{53}y - \frac{3566808}{53}x^{2},\\
	A_{9}^{(2)} &= -\frac{142288}{29} + \frac{666400}{29}x
	- \frac{3104640}{29}y - \frac{551712}{29}x^{2},\\
	A_{10}^{(2)} &= -\frac{65670}{61} + \frac{108150}{61}x
	- \frac{2976750}{61}y - \frac{35280}{61}x^{2}
	+ \frac{2192400}{61}x y.
\end{align*}
Using the Maple implementation, biorthogonality can be verified:
\[
\int_{\Delta} \left( B(\boldsymbol{x}) \right)^{[10]}\, \d \mu(\boldsymbol x) \left( A(\boldsymbol{x}) \right)^{[10]}  = I_{10}.
\]
\item 
We now turn to the second example. With \(q=2,\, p=2\) and the previously specified parameters, the first matrix polynomials associated with \(B(x)\) are
\begin{align*}
    B_{(0,0)}(\boldsymbol{x})  & = \begin{bNiceMatrix}
        1 & 0 
\\
 -\frac{256}{1155} & 1 
    \end{bNiceMatrix}, & B_{(1,0)}(\boldsymbol{x}) & = \begin{bNiceMatrix}
        \frac{-\mathscr{C}_1}{5}\left(-11+5 \pi\right)+x  & \mathscr{C}_1 \frac{231 \pi}{256}  
\\[4pt]
 \mathscr{C}_2\left( \frac{1024 \pi}{273}-\frac{4096}{385}-\frac{2048}{3003}\left(-65+22 \pi \right)x \right) & -2\mathscr{C}_2 \left(9 \pi -26\right)+x
    \end{bNiceMatrix},
\end{align*}
\begin{multline*}
     B_{(1,1)}(\boldsymbol{x}) =  \begin{bNiceMatrix}
        -\frac{1}{2}+\frac{x}{2}+y  & 0 
\\[12pt]
 \mathscr{C}_3 \left( \frac{1024}{1001} \pi^{2}-\frac{241408}{45045} \pi +\frac{24064}{3465}-\frac{512}{45045} \left(90 \pi^{2}-467 \pi +598\right)x- \right. & -\frac{\mathscr{C}_3}{6}\left(30 \pi^{2}-159 \pi +208\right)+
\\
\left. \frac{512}{45045} \left(180 \pi^{2}-959 \pi +1261\right) y \right) & \frac{\mathscr{C}_3}{12}\left(60 \pi^{2}-319 \pi +416\right)x+y
    \end{bNiceMatrix},
\end{multline*}
whit constants,
\[\begin{aligned}
    \mathscr{C}_1 & = \frac{1}{-11+4 \pi }, & \mathscr{C}_2 & = \frac{1}{-286+95 \pi }, & \mathscr{C}_3 & = \frac{1}{\left(11 \pi -26\right) \left(-3+\pi \right)}.
\end{aligned}\]
The first three matrix polynomials associated with $A(\boldsymbol{x})$ are: 
\[
    A_{(0,0)}(\boldsymbol{x}) = \begin{bNiceMatrix} 
    24 & \mathscr{C}_1\frac{3465 \pi}{8} \\[4pt]
    0 & -\mathscr{C}_1\frac{3465}{8}  
    \end{bNiceMatrix},
\]
\[
    A_{(1,0)}(\boldsymbol{x}) = \begin{bNiceMatrix} 
    \mathscr{C}_2\left(-360 \left(40 \pi -143\right)+23400 \left(-11+4 \pi \right)x\right) & - \mathscr{C}_3\frac{45045 \pi}{128}\left(-91+30 \pi \right)+\frac{8783775 \pi  }{128 \left(11 \pi -26\right)}x \\[4pt]
    - 4320\mathscr{C}_2 & \mathscr{C}_3 \left( -\frac{195195}{8}+\frac{255255 \pi}{32}-\frac{15015}{32} \left(-286+95 \pi \right)x \right) 
    \end{bNiceMatrix},
\]
\[
    A_{(1,1)}(\boldsymbol{x}) = \begin{bNiceMatrix} 
    -48 \mathscr{C}_3\left(75 \pi^{2}-424 \pi +585\right)+  & \mathscr{C}_4 \left( \frac{135135}{32} \pi  \left(225 \pi^{2}-1217 \pi +1625\right)+ \right. 
\\
 \frac{720 \left(-13+5 \pi \right)}{11 \pi -26}x+720 y & \left. \frac{6081075}{32} \left(-13+5 \pi \right) \pi  \left(-3+\pi \right)x+\frac{6081075\pi}{32 \mathscr{C}_3}y \right)\\[12pt]
    \mathscr{C}_3\left( -64 \left(7 \pi -17\right)+ \right. & \mathscr{C}_4 \left( \frac{45045}{8} \left(144 \pi^{2}-779 \pi +1040\right) -  \right. \\
    \left. 128 \left(-13+5 \pi \right)x \right) & \left. \frac{45045}{8} \left(135 \pi^{2}-752 \pi +1040\right)x - \frac{675675}{4 \mathscr{C}_3}y \right)
    \end{bNiceMatrix},
\]
where,
\[
    \mathscr{C}_4
 = \frac{1}{540 \pi^{3}-4362 \pi^{2}+11693 \pi -10400}.
\]
\end{enumerate}
\section*{Conclusions and outlook}

In this work, we have introduced a framework for studying mixed-type multiple orthogonal polynomials in two variables along the step-line, based on the Gauss–Borel factorization of a suitably constructed moment matrix \cite{manas}. Inspired by the recent contribution of Fernández and Villegas \cite{Lidia}, where multiple orthogonal polynomials in two variables are studied over the full lattice, our approach focuses on the mixed-type case and leads naturally to explicit orthogonality relations, recurrence formulas, Christoffel–Darboux-type identities, and an ABC-type theorem adapted to the step-line configuration. The combination of algebraic techniques and structural analysis allows for a transparent formulation of the key properties of these bivariate polynomial families.
Several directions remain open for further exploration.

 A natural extension of this work concerns the spectral analysis of Christoffel, Geronimus, and Uvarov transformations in the mixed-type, bivariate setting. In this context, one may follow the line developed in \cite{AiM_multivariable} and \cite{JAT_multivariable}, where spectral transformations for multivariate matrix orthogonal polynomials were investigated. There, the bigraded nature of the matrix polynomial theory adds substantial complexity and possibly requires the use of tools from algebraic geometry. Translating and adapting such spectral techniques to the mixed-type case presents a challenging but promising direction.

Another relevant avenue concerns the formulation of Favard-type theorems in this setting. On the one hand, it would be of interest to characterize non-spectral Favard-type results in terms of linear functionals and their moment structure. On the other hand, one could aim to establish spectral Favard-type theorems based on the existence of factorizations of truncated recurrence matrices into products of bidiagonal matrices. In particular, a growing number of bidiagonal factors might be needed to reflect the structural complexity of the step-line case, with the ultimate goal of guaranteeing the oscillatory nature of each truncation. In this direction, we refer to our recent developments in \cite{aim,phys-scrip,Contemporary}, where spectral Favard-type theory for mixed-type multiple orthogonality has been established via positive bidiagonal factorizations.

Finally, another promising direction lies in the connection between mixed-type multiple orthogonality and integrable systems of Toda type. As shown in \cite{afm}, the moment matrix structure and its Gauss–Borel factorization naturally lead to discrete flows governed by multi-component Toda equations. These connections, already present in the univariate and block matrix settings, become significantly richer in the multivariate case, as explored in \cite{AiM_multivariable}. Extending this integrable systems perspective to the mixed-type, bivariate setting—where orthogonality conditions, recurrence relations, and transformations interact in a highly structured way—could lead to the identification of new integrable hierarchies.

	\section*{Acknowledgments}
%The authors extend their gratitude to the anonymous reviewer whose insightful comments and suggestions have significantly enhanced the paper's presentation.
%MR gratefully acknowledges support from a \emph{Beca de Colaboración} awarded by the Complutense University of Madrid.

The authors MMB and MR acknowledges research projects [PID2021- 122154NB-I00], \emph{Ortogonalidad y Aproximación con Aplicaciones en Machine Learning y Teoría de la Probabilidad}  funded  by
\href{https://doi.org/10.13039/501100011033}{MICIU/AEI/10.13039 /501100011033} and by "ERDF A Way of making Europe” and  [PID2024-155133NB-I00],  \emph{Ortogonalidad, aproximación e integrabilidad: aplicaciones en procesos estocásticos clásicos y cuánticos}.

%We are grateful to the reviewers for their careful reading and insightful suggestions, which helped improve the clarity and quality of the manuscript.

\section*{Declarations}

\begin{enumerate}
	\item \textbf{Conflict of interest:} The authors declare that they have no conflict of interest.
	\item \textbf{Ethical approval:} Not applicable.
	\item \textbf{Author contributions:} MM formulated the problem and developed the methodology. MR, with the support of MM, established the main results. JW contributed to the discussions and worked out an illustrative example.
	\item \textbf{Data availability:}
The Maple code used to generate the bivariate
Jacobi--Pi\~neiro multiple orthogonal polynomials of mixed type on the
triangle is openly available at \url{https://github.com/ManuelManas/Bivariate-Jacobi--Pineiro-of-the-Mixed-Type-on-the-Triangle}.
%\href{https://github.com/ManuelManas/Bivariate-Jacobi--Pineiro-of-the-Mixed-Type-on-the-Triangle}
%{\texttt{ManuelManas/Bivariate-Jacobi--Pineiro} (GitHub)}.
\end{enumerate}

\end{document}